\documentclass[a4paper,11pt]{amsart}
\usepackage{mathrsfs}
\usepackage{cases}
\usepackage{amsfonts}
\usepackage{graphicx}
\usepackage{amsmath}
\usepackage{amssymb}
\usepackage[all]{xypic}
\usepackage[all]{xy}
\usepackage{color}
\usepackage{colordvi}
\usepackage{multicol}

\begin{document}
\input xy
\xyoption{all}

\renewcommand{\mod}{\operatorname{mod}\nolimits}
\newcommand{\proj}{\operatorname{proj.}\nolimits}
\newcommand{\inj}{\operatorname{inj.}\nolimits}
\newcommand{\rad}{\operatorname{rad}\nolimits}
\newcommand{\soc}{\operatorname{soc}\nolimits}
\newcommand{\ind}{\operatorname{inj.dim}\nolimits}
\newcommand{\Ginj}{\operatorname{Ginj}\nolimits}
\newcommand{\Mod}{\operatorname{Mod}\nolimits}
\newcommand{\R}{\operatorname{R}\nolimits}
\newcommand{\End}{\operatorname{End}\nolimits}
\newcommand{\ob}{\operatorname{Ob}\nolimits}
\newcommand{\Ht}{\operatorname{Ht}\nolimits}
\newcommand{\cone}{\operatorname{cone}\nolimits}
\newcommand{\rep}{\operatorname{rep}\nolimits}
\newcommand{\Ext}{\operatorname{Ext}\nolimits}
\newcommand{\Tor}{\operatorname{Tor}\nolimits}
\newcommand{\Hom}{\operatorname{Hom}\nolimits}
\newcommand{\Pic}{\operatorname{Pic}\nolimits}
\newcommand{\Coker}{\operatorname{Coker}\nolimits}
\newcommand{\Div}{\operatorname{Div}\nolimits}
\newcommand{\rank}{\operatorname{rank}\nolimits}
\newcommand{\Gproj}{\operatorname{Gproj}\nolimits}
\newcommand{\Len}{\operatorname{Length}\nolimits}
\newcommand{\RHom}{\operatorname{RHom}\nolimits}
\renewcommand{\deg}{\operatorname{deg}\nolimits}
\renewcommand{\Im}{\operatorname{Im}\nolimits}
\newcommand{\Ker}{\operatorname{Ker}\nolimits}
\newcommand{\Coh}{\operatorname{Coh}\nolimits}
\newcommand{\Id}{\operatorname{Id}\nolimits}
\newcommand{\Qcoh}{\operatorname{Qch}\nolimits}
\newcommand{\CM}{\operatorname{CM}\nolimits}
\newcommand{\Cp}{\operatorname{Cp}\nolimits}
\newcommand{\coker}{\operatorname{Coker}\nolimits}
\renewcommand{\dim}{\operatorname{dim}\nolimits}
\renewcommand{\div}{\operatorname{div}\nolimits}
\newcommand{\Ab}{{\operatorname{Ab}\nolimits}}
\renewcommand{\Vec}{{\operatorname{Vec}\nolimits}}
\newcommand{\pd}{\operatorname{proj.dim}\nolimits}
\newcommand{\id}{\operatorname{inj.dim}\nolimits}
\newcommand{\Gd}{\operatorname{G.dim}\nolimits}
\newcommand{\sdim}{\operatorname{sdim}\nolimits}
\newcommand{\add}{\operatorname{add}\nolimits}
\newcommand{\pr}{\operatorname{pr}\nolimits}
\newcommand{\oR}{\operatorname{R}\nolimits}
\newcommand{\oL}{\operatorname{L}\nolimits}
\newcommand{\Perf}{{\mathfrak Perf}}
\newcommand{\cc}{{\mathcal C}}
\newcommand{\ce}{{\mathcal E}}
\newcommand{\cs}{{\mathcal S}}
\newcommand{\cf}{{\mathcal F}}
\newcommand{\cx}{{\mathcal X}}
\newcommand{\ct}{{\mathcal T}}
\newcommand{\cu}{{\mathcal U}}
\newcommand{\cv}{{\mathcal V}}
\newcommand{\cn}{{\mathcal N}}
\newcommand{\ch}{{\mathcal H}}
\newcommand{\ca}{{\mathcal A}}
\newcommand{\cb}{{\mathcal B}}
\newcommand{\ci}{{\mathcal I}}
\newcommand{\cj}{{\mathcal J}}
\newcommand{\cm}{{\mathcal M}}
\newcommand{\cp}{{\mathcal P}}
\newcommand{\cg}{{\mathcal G}}
\newcommand{\cw}{{\mathcal W}}
\newcommand{\co}{{\mathcal O}}
\newcommand{\cd}{{\mathcal D}}
\newcommand{\ck}{{\mathcal K}}
\newcommand{\calr}{{\mathcal R}}
\newcommand{\ol}{\overline}
\newcommand{\ul}{\underline}
\newcommand{\st}{[1]}
\newcommand{\ow}{\widetilde}
\renewcommand{\P}{\mathbf{P}}
\newcommand{\pic}{\operatorname{Pic}\nolimits}
\newcommand{\Spec}{\operatorname{Spec}\nolimits}
\newtheorem{theorem}{Theorem}[section]
\newtheorem{acknowledgement}[theorem]{Acknowledgement}
\newtheorem{algorithm}[theorem]{Algorithm}
\newtheorem{axiom}[theorem]{Axiom}
\newtheorem{case}[theorem]{Case}
\newtheorem{claim}[theorem]{Claim}
\newtheorem{conclusion}[theorem]{Conclusion}
\newtheorem{condition}[theorem]{Condition}
\newtheorem{conjecture}[theorem]{Conjecture}
\newtheorem{construction}[theorem]{Construction}
\newtheorem{corollary}[theorem]{Corollary}
\newtheorem{criterion}[theorem]{Criterion}
\newtheorem{definition}[theorem]{Definition}
\newtheorem{example}[theorem]{Example}
\newtheorem{exercise}[theorem]{Exercise}
\newtheorem{lemma}[theorem]{Lemma}
\newtheorem{notation}[theorem]{Notation}
\newtheorem{problem}[theorem]{Problem}
\newtheorem{proposition}[theorem]{Proposition}
\newtheorem{remark}[theorem]{Remark}
\newtheorem{solution}[theorem]{Solution}
\newtheorem{summary}[theorem]{Summary}
\newtheorem*{thm}{Theorem}

\def \bp{{\mathbf p}}
\def \bA{{\mathbf A}}
\def \bL{{\mathbf L}}
\def \bF{{\mathbf F}}
\def \bS{{\mathbf S}}
\def \bC{{\mathbf C}}

\def \Z{{\Bbb Z}}
\def \F{{\Bbb F}}
\def \C{{\Bbb C}}
\def \N{{\Bbb N}}
\def \Q{{\Bbb Q}}
\def \G{{\Bbb G}}
\def \P{{\Bbb P}}
\def \K{{\Bbb K}}
\def \E{{\Bbb E}}
\def \A{{\Bbb A}}
\def \BH{{\Bbb H}}
\def \T{{\Bbb T}}
\newcommand {\lu}[1]{\textcolor{red}{$\clubsuit$: #1}}

\title[singularity categories of the Cluster-tilted algebras]{The singularity categories of the Cluster-tilted algebras of Dynkin type}

\author[Chen]{Xinhong Chen}
\address{Department of Mathematics, Southwest Jiaotong University, Chengdu 610031, P.R.China}
\email{chenxh2007@aliyun.com}

\author[Geng]{Shengfei Geng}
\address{Department of Mathematics, Sichuan University, Chengdu 610064, P.R.China}
\email{genshengfei@scu.edu.cn}

\author[Lu]{Ming Lu$^\dag$}
\address{Department of Mathematics, Sichuan University, Chengdu 610064, P.R.China}
\email{luming@scu.edu.cn}
\thanks{$^\dag$ Corresponding author}

\subjclass[2000]{16E45, 18E30, 18E35}
\keywords{cluster-tilted algebras; good mutation; Gorenstein algebras; Gorenstein projective modules; singularity category}

\begin{abstract}We use the stable categories of some selfinjective algebras to describe the singularity categories of the cluster-tilted algebras of Dynkin type. Furthermore, in this way, we settle the problem of singularity equivalence classification of the cluster-tilted algebra of type $A$, $D$ and $E$ respectively.
\end{abstract}

\maketitle

\section{Introduction}
Cluster algebras have been introduced by Fomin and Zelevinsky around 2000 \cite{FZ1}. In attempt to categorify cluster algebras, cluster categories have been introduced by Buan, Marsh, Reiten, Reineke and Todorov in \cite{BMRRT}. More precisely, the cluster category is associated with a finite dimensional hereditary algebra $H$ over a field $K$, and defined as the quotient $\cc_H$ of the bounded derived category $D^b(\mod H)$ of finitely generated modules over $H$ by the functor $\Sigma \tau^{-1}$, where $\tau$ denotes the AR-translation and $\Sigma$ denotes the shift functor. The cluster category is triangulated \cite{Ke}, see also \cite{CCS1} for the cluster category of type $A_n$. In fact, the cluster categories are Calabi-Yau of dimension two \cite{BMRRT,Ke}.

Importantly, in the categorification of cluster algebras, cluster tilting objects categorify clusters of the corresponding cluster algebras, and the combinatorics structure of cluster tilting objects is the same as the combinatorics structure of the corresponding cluster algebras \cite{CC,CK}.

For a cluster tilting object $T$, its endomorphism algebra $\Gamma=\End_{\cc_H}(T)^{op}$ is called the cluster-tilted algebra \cite{BMR1}. The module category of finitely generated $\Gamma$-modules was explicitly described in \cite{BMR1}. In particular, the cluster-tilted algebra is a path algebra of a quiver with relations, and it was proved that a cluster-tilted algebra of finite type is uniquely determined by its quiver \cite{BMR2}. Moreover, the relations were explicitly described in \cite{CCS2} for a cluster-tilted algebra of finite type. In fact, the quivers of the cluster-tilted algebras of Dynkin type $Q$ are precisely the ones obtained from $Q$ by performing finitely many quiver mutations.

The derived equivalence classification of cluster-tilted algebras of Dynkin type has been settled by Buan and Vatne for cluster-tilted algebras of type $A_n$ \cite{BV}, and by Bastian, Holm and Ladkani for cluster-tilted algebras of type $E_6,E_7,E_8$\cite{BHL1}. It turns out that two cluster-tilted algebras of type $A_n$ are derived equivalent if and only if their quivers have the same number of $3$-cycles;
Two cluster-tilted algebra of Dynkin type $E$ are derived equivalent if and only if their Cartan matrices represent equivalent bilinear forms over $\Z$. This in turn happens if and only if the Cartan matrices of the algebras have the same determinant and the same characteristic polynomial of their asymmetry matrices.
For Dynkin type $D_n$, Bastian, Holm and Ladkani consider the mutations of quivers which preserve derived equivalences, and get a far reaching derived equivalence classification and suggest standard forms for the derived equivalence classes \cite{BHL2}.

For an algebra, we know that the singularity category is also an important triangulated category. The singularity category of an algebra is defined to be the Verdier quotient of the bounded derived category with respect to the thick subcategory formed by complexes isomorphic to bounded complexes of finitely generated projective modules \cite{Bu,Ha1,Or1}. In this paper, we address the problem of singularity equivalence classification of the cluster-tilted algebra of Dynkin type. Furthermore, it was proved by Keller and Reiten that the cluster-tilted algebras are Gorenstein of dimension at most $1$ \cite{KR}. A fundamental result of Buchweitz \cite{Bu} and Happel \cite{Ha1} states that for a Gorenstein algebra $A$, the singularity category is triangle equivalent to the stable category of (maximal) Cohen-Macaulay (also called Gorenstein projective) $A$-modules.

In this paper, we find some selfinjective algebra, whose stable category is triangulated equivalent to the singularity category of the cluster-tilted algebra of Dynkin type. For cluster-tilted algebras of type $A$, since they are gentle algebras, it was settled by Kalck in \cite{Ka}. We use the good mutation defined in \cite{BHL2} to give a direct proof, see Corollary \ref{corollary CM module of cluster-tilted algebra of type A}. For type $D$, using the description of its quiver by Geng and Peng \cite{GP}, Vatne \cite{V}, we get some selfinjective algebras, whose stable categories are equivalent to the singularity categories of each type, see Theorem \ref{theorem singularity category of D}. For type $E$, use the derived equivalence classification in \cite{BHL1}, we get
that two cluster-tilted algebras of Dynkin type $E$ are singularity equivalent if and only if the Cartan matrices of the algebras have the same determinant and the
same trace of their asymmetry matrices, see Theorem \ref{theorem singularity category of E}.

\vspace{0.2cm} \noindent{\bf Acknowledgments.}
This work is inspired by some discussions with Changjian Fu. The authors thank him very much. The authors also thank the referee for very helpful and insightful comments.

The first author(X. Chen) was supported by the Fundamental Research Funds for the Central Universities A0920502051411-45.

\section{Preliminaries}
In this paper, we denote by $S_l$ the self-injective Nakayama algebra of a cyclic quiver with $l$ vertices modulo the ideal generated by paths of length $l-1$.
For a cluster-tilted algebra $KQ/I$ of type $A$, we denote by $t(Q)$ the number of the triangles in $Q$.
\subsection{Mutation of algebras}
We recall the notion of mutations of algebras from \cite{La}. Let $\Gamma=KQ/I$ be an algebra given as a quiver with relations. For any vertex $i$ of $Q$, there is a trivial path $e_i$ of length $0$; the corresponding indecomposable projective module $P_i=e_i\Gamma$ is spanned by the images of the paths starting at $i$. Thus an arrow $i\xrightarrow{\alpha}j$ gives rise to a map $P_j\rightarrow P_i$ given by left multiplication with $\alpha$.

Let $k$ be a vertex of $Q$ without loops. Consider the following two complexes of projective $\Gamma$-modules
$$T_k^{-}(\Gamma)=(P_k\xrightarrow{f} \bigoplus_{j\rightarrow k}P_j)\oplus(\bigoplus_{i\neq k}P_i), T_k^+(\Gamma)=(\bigoplus_{k\rightarrow j} P_j\xrightarrow{g} P_k )\oplus (\bigoplus_{i\neq k}P_i) $$
where the map $f$ is induced by all the maps $P_k\rightarrow P_j$ corresponding to the arrow $j\rightarrow k$ ending at $k$, the map $g$ is induced by the maps $P_j\rightarrow P_k$ corresponding to the arrow $k\rightarrow j$ starting at $k$, the term $P_k$ lies in degree $-1$ in $T^{-}_k(\Gamma)$ and in degree $1$ in
$T_k^+(\Gamma)$, and all other terms are in degree $0$.

\begin{definition}\cite{BHL2}
Let $\Gamma$ be an algebra given as a quiver with relations and $k$
a vertex without loops.

(a) We say that the negative mutation of $\Gamma$ at $k$ is defined if $T_k^-(\Gamma)$ is a tilting complex over $\Gamma$. In this case, we call the algebra
$\mu_k^-(\Gamma)=\End_{D^b(A)}(T^-_k(\Gamma))$ the negative mutation of $\Gamma$ at the vertex $k$.

(b) We say that the positive mutation of $\Gamma$ at $k$ is defined if $T_k^+(\Gamma)$ is a tilting complex over $\Gamma$. In this case, we call the algebra
$\mu_k^+(\Gamma)=\End_{D^b(\Gamma)}(T^+_k(\Gamma))$ the positive mutation of $\Gamma$ at the vertex $k$.
\end{definition}

Given a quiver $Q$ without loops and 2-cycles, and a vertex $k$, we denote by $\mu_k(Q)$ the Fomin-Zelevinsky quiver mutation of $Q$ at $k$. Two quivers are called \emph{mutation equivalent} if one can be reached from the other by a finite sequence of quiver mutations. We also denote by $\Lambda_Q$ the corresponding cluster-tilted algebra.

\begin{proposition}\cite{BHL2}\label{proposition definition good mutation }
Let $Q$ be mutation equivalent to a Dynkin quiver and let $k$ be a vertex of $Q$.

(a) $\Lambda_{\mu_k(Q)}\simeq\mu_k^-(\Lambda_Q)$ if and only if the two algebra mutations $\mu_k^-(\Lambda_Q)$ and $\mu_k^+(\Lambda_{\mu_k(Q)})$ are defined.

(b) $\Lambda_{\mu_k(Q)}\simeq\mu_k^+(\Lambda_Q)$ if and only if the two algebra mutations $\mu_k^+(\Lambda_Q)$ and $\mu_k^-(\Lambda_{\mu_k(Q)})$ are defined.

\end{proposition}
\begin{definition}\cite{BHL2}
When (at least) one of the conditions in the proposition holds, we say that the quiver mutation of $Q$ at $k$ is good, since it implies the derived equivalence of the corresponding cluster-tilted algebra $\Lambda_Q$ and $\Lambda_{\mu_k(Q)}$.
\end{definition}

\subsection{Gorenstein algebra and Gorenstein projective module}
Let $\Gamma$ be a finite-dimensional $K$-algebra. Let $\mod \Gamma$ be the category of finitely generated left $\Gamma$-modules. By $D=\Hom_K(-,K)$ we denote the standard duality with respect to the ground field. Then $_\Gamma D(\Gamma_\Gamma)$ is an injective cogenerator for $\mod \Gamma$. For an arbitrary $\Gamma$-module $_\Gamma X$, we denote by $\pd_\Gamma X$ (resp. $\id_\Gamma X$) the projective dimension (resp. the injective dimension) of the module $_\Gamma X$. An $\Gamma$-module $G$ is \emph{Gorenstein projective}, if there is an exact sequence $$P^\bullet:\cdots \rightarrow P^{-1}\rightarrow P^0\xrightarrow{d^0}P^1\rightarrow \cdots$$ of projective $\Gamma$-modules, which stays exact under $\Hom_\Gamma(-,\Gamma)$, and such that $G\cong \Ker d^0$. We denote by $\Gproj(\Gamma)$ the subcategory of Gorenstein projective modules.

\begin{definition}\cite{AR1,AR2,Ha1}
A finite dimensional algebra $\Gamma$ is called a Gorenstein algebra if $\Gamma$ satisfies $\pd_\Gamma D(\Gamma_\Gamma)<\infty$ and $\id_\Gamma \Gamma<\infty$.
\end{definition}

Observe that for a Gorenstein algebra $\Gamma$, we have $\id _\Gamma\Gamma=\id \Gamma_\Gamma$, \cite[Lemma 6.9]{Ha1}; the common value is denoted by $\Gd \Gamma$. If $\Gd \Gamma\leq d$, we say that $\Gamma$ is $d$-Gorenstein.

Let $\cx$ be a subcategory of $\mod \Gamma$. Then $^\bot\cx:=\{M|\Ext^i(M,X)=0, \mbox{ for all } X\in\cx, i\geq1\}$. In particular, we define $^\bot \Gamma:=^\bot (\mathrm{proj.} \Gamma)$, where $\mathrm{proj.}\Gamma$ is the category of finite generated projective $\Gamma$-modules.

\begin{theorem}\cite{AM,EJ}\label{theorem characterize of gorenstein property}
Let $\Gamma$ be an artin algebra and let $d\geq0$. Then the following statements are equivalent:

(1) the algebra $\Gamma$ is $d$-Gorenstein;

(2) $\Gproj(\Gamma)=\Omega^d(\mod \Gamma)$, where $\Omega$ is the syzygy functor.

In this case, A module $G$ is Gorenstein projective if and only if there is an exact sequence $0\rightarrow G\rightarrow P^0\rightarrow P^1\rightarrow \cdots$ with each
$P^i$ projective, we have $\Gproj(\Gamma)= ^\bot \Gamma$.
\end{theorem}

By Theorem \ref{theorem characterize of gorenstein property}, we get the following simple observation for $1$-Gorenstein algebras. Indeed, almost all algebras in this paper are $1$-Gorenstein.

\begin{remark}
If $\Gamma$ is $1$-Gorenstein, then the Gorenstein projective modules are just torsionless modules.
\end{remark}
\begin{proof}
It is easy to see that the Gorenstein projective modules are torsionless since torsionless modules are submodules of projective modules. On the other hand, for any torsionless module $M$, we get an injective morphism $M\xrightarrow{f} P$, where $P$ is projective, which is complete to a short exact sequence
$$0\rightarrow M\xrightarrow{f}P\xrightarrow{g}N\rightarrow0,$$
it implies $M\in\Omega(\Gamma)$. By Theorem \ref{theorem characterize of gorenstein property}, we get that $M$ is Gorenstein projective since $\Gamma$ is $1$-Gorenstein.
\end{proof}

For an algebra $\Gamma$, the \emph{singularity category} of $\Gamma$ is defined to be the quotient category $D_{sg}^b(\Gamma):=D^b(\Gamma)/K^b(\proj \Gamma)$ \cite{Bu,Ha1,Or1}.
For any two algebras, if their singularity categories are equivalent, then we call them to be \emph{singularity equivalent}.

\begin{theorem}\cite{Bu,Ha1}
Let $\Gamma$ be a Gorenstein algebra. Then $\Gproj (\Gamma)$ is a Frobenius category with the projective modules as the projective-injective objects. The stable category $\underline{\Gproj}(\Gamma)$ is triangulated equivalent to the singularity category $D^b_{sg}(\Gamma)$ of $\Gamma$.
\end{theorem}

For this reason, we do not distinguish the singularity category and the stable category of the Gorenstein projective modules for a Gorenstein algebra.

\section{Rooted quiver of type $A$}

Recall that the quiver $A_n$ is the following direct graph on $n\geq1$ vertices
$$\circ_1\rightarrow\circ_2 \rightarrow \cdots \rightarrow \circ_n.$$

The neighborhood of a vertex $x$ in a quiver $Q$ is the full subquiver of $Q$ on the subset of vertices consisted of $x$ and the vertices which are targets of arrows starting at $x$ or sources of arrows ending at $x$.
\setlength{\unitlength}{0.8mm}
\begin{center}
\begin{picture}(100,30)

\put(0,20){\circle{1.3}}

\put(2,18){\vector(1,-1){7}}
\put(10,10){\circle*{1.3}}

\put(40,20){\circle{1.3}}

\put(50,10){\circle*{1.3}}
\put(48,12){\vector(-1,1){7}}

\put(80,0){\circle{1.3}}
\put(80,20){\circle{1.3}}

\put(90,10){\circle*{1.3}}
\put(82,18){\vector(1,-1){7}}
\put(88,8){\vector(-1,-1){7}}
\put(80,2){\vector(0,1){16}}
\end{picture}
\end{center}

\setlength{\unitlength}{0.8mm}
\begin{center}
\begin{picture}(100,30)
\put(0,0){\circle{1.3}}
\put(0,20){\circle{1.3}}
\put(2,2){\vector(1,1){7}}
\put(2,18){\vector(1,-1){7}}
\put(10,10){\circle*{1.3}}

\put(40,0){\circle{1.3}}
\put(40,20){\circle{1.3}}

\put(50,10){\circle*{1.3}}
\put(48,12){\vector(-1,1){7}}
\put(48,8){\vector(-1,-1){7}}

\put(80,0){\circle{1.3}}
\put(80,20){\circle{1.3}}

\put(90,10){\circle*{1.3}}
\put(82,18){\vector(1,-1){7}}
\put(88,8){\vector(-1,-1){7}}

\end{picture}
\end{center}

\setlength{\unitlength}{0.8mm}
\begin{center}
\begin{picture}(100,30)
\put(0,0){\circle{1.3}}
\put(2,0){\vector(1,0){16}}
\put(20,0){\circle{1.3}}

\put(18,2){\vector(-1,1){7}}
\put(10,10){\circle*{1.3}}
\put(8,8){\vector(-1,-1){7}}

\put(0,20){\circle{1.3}}
\put(2,18){\vector(1,-1){7}}

\put(40,0){\circle{1.3}}
\put(42,2){\vector(1,1){7}}
\put(50,10){\circle*{1.3}}
\put(48,12){\vector(-1,1){7}}
\put(40,20){\circle{1.3}}
\put(40,18){\vector(0,-1){16}}
\put(52,8){\vector(1,-1){7}}
\put(60,0){\circle{1.3}}

\put(80,0){\circle{1.3}}
\put(82,2){\vector(1,1){7}}
\put(90,10){\circle*{1.3}}
\put(88,12){\vector(-1,1){7}}
\put(80,20){\circle{1.3}}
\put(80,18){\vector(0,-1){16}}
\put(92,8){\vector(1,-1){7}}
\put(100,0){\circle{1.3}}
\put(100,2){\vector(0,1){16}}
\put(100,20){\circle{1.3}}
\put(98,18){\vector(-1,-1){7}}

\end{picture}
\end{center}
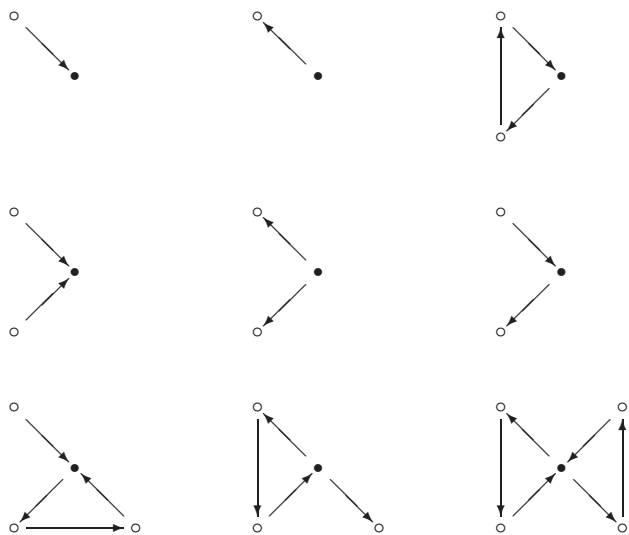
Figure 1. The 9 possible neighborhood of a vertex $\bullet$ in a quiver which is mutation equivalent to $A_n$, $n\geq2$. The three at the first row are the possible neighborhood of a root in a rooted quiver of type $A$.

\begin{lemma}\cite{BHL2}
 Let $n\geq2$. A quiver is mutation equivalent to $A_n$ if and only if it has $n$ vertices, the
neighborhood of each vertex is one of the nine depicted in Figure 1, and there are no cycles in its underlying
graph apart from those induced by oriented cycles contained in neighborhoods of vertices.
\end{lemma}
\begin{remark}\cite{BMR2,CCS1,CCS2}\label{remark cluster-tilted algebra of A}
Given a quiver $Q$ mutation equivalent to $A_n$, the relations defining the corresponding cluster-tilted algebra $\Lambda_Q$ (which has $Q$ as its quiver) are obtained as follows: any triangle
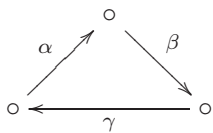
\[\xymatrix{&\circ\ar[dr]^\beta&\\
\circ\ar[ur]^\alpha &&\circ\ar[ll]^{\gamma}}\]
in $Q$ gives rise to three zero relations $\beta\alpha,\gamma\beta,\alpha\gamma$, and there are no other relations.
\end{remark}
\begin{definition}\cite{BHL2}
A rooted quiver of type $A$ is a pair $(Q,v)$ where $Q$ is a quiver which is mutation equivalent to $A_n$ and $v$ is a vertex of $Q$(the root) whose neighborhood is one of the three in
the first row of Figure 1 if $n\geq2$.

Let $Q_0$ be a quiver, called a skeleton, and let $c_1,c_2,\dots,c_k$ be $k\geq0$ distinct vertices of $Q_0$. The gluing of $k$ rooted quiver of type $A$, say
$$(Q_1,v_1),(Q_2,v_2),\dots,(Q_k,v_k),$$ to $Q_0$ at the vertices $c_1,\dots,c_k$ is defined as the quiver $Q$ obtained from the distinct union $Q_0\bigsqcup Q_1\bigsqcup\cdots \bigsqcup Q_k$ by identifying each vertex $c_i$ with the corresponding root $v_i$, for $1\leq i\leq k$.
\end{definition}
\begin{remark}\cite{BHL2}
Given relations $I_0$ on the skeleton $Q_0$, they induce relations $I$ on the gluing quiver $Q$, namely by taking the union of all the relations on $Q_0,Q_1,\dots,Q_k$, where the relations $I_0,I_1,\dots,I_k$ on the rooted quivers of type $A$ are those stated in Remark \ref{remark cluster-tilted algebra of A}. In particular, we also denote the quotient algebra by $\Lambda_Q=KQ/I$.
\end{remark}

For a rooted quiver $(Q,v)$ of type $A$, we call a mutation at a vertex other than the root $v$ a mutation outside the root.

The standard form of a rooted quiver $(Q,v)$ of type $A$ is a rooted quiver of type $A$ as in the following diagram, consisting of $s(Q)$ lines and $t(Q)$ triangles with the vertex $v$ as the root.
\[\xymatrix {&&&  &\circ \ar[dl]  &&&  \circ \ar[dl]  &\\
\bullet_v\ar[r] &\circ \ar[r]& \cdots \ar[r] &\circ \ar[r] &\circ\ar[u] & \cdots& \circ \ar[r] & \circ \ar[u]    }\]

\begin{proposition}\cite{BHL1}\label{proposition rooted quiver of type A}
Any two rooted quivers of type $A$ with the same numbers of lines and triangles can be connected by a sequence of good mutations outside the root.
\end{proposition}

\begin{proposition}\label{proposition Gorenstein algebra under good mutation}
Let $Q$ be the gluing of $k$ rooted quivers of type $A$, say
$(Q_1,v_1)$, $(Q_2,v_2),\dots,(Q_k,v_k)$, to $Q_0$ at the distinct vertices $c_1,\dots,c_k$. If $\mu_j$ is a good mutation for any $j \in Q_i$ other than $v_i$, then $\Lambda_{\mu_j(Q)}$ is derived equivalent to $\Lambda_{Q}$. In particular, $\Lambda_{\mu_j(Q)}$ is a Gorenstein algebra if and only if so is $\Lambda_{Q}$.
\end{proposition}
\begin{proof}
By Proposition \ref{proposition definition good mutation } and the definition of good mutation, we know that $\Lambda_{\mu_j(Q)}$ is derived equivalent to $\Lambda_{Q}$.

By \cite[Theorem 6.4 ]{Ric}, we get that $K^b(\proj\Lambda_{\mu_j(Q)})$ is equivalent to $K^b(\proj\Lambda_{Q})$, and by \cite[Proposition 9.1]{Ric}, we get that
$K^b(\inj\Lambda_{\mu_j(Q)})$ is equivalent to $K^b(\inj\Lambda_{Q})$. \cite[Lemma 1.5]{Ha1} shows that for any algebra $A$, $A$ is a Gorenstein algebra if and only if
$K^b(\proj A)=K^b(\inj A)$. So $\Lambda_{\mu_j(Q)}$ is a Gorenstein algebra if and only if so is $\Lambda_{Q}$.
\end{proof}

Let $B$ and $C$ be Gorenstein algebras, $M$ a $(B,C)$-bimodule and $\Lambda=\left(\begin{array}{cc}B &M\\0&¡¡C\end{array}\right)$ be the upper triangular matrix algebra. Recall that $\Lambda$ is Gorenstein if and only if the bimodule $M$ has finite projective dimension both as left $B$-module and as right $C$-module \cite[Theorem 3.3]{Chen1}.

Let $B=KQ'/I'$ and $C=KQ''/I''$. Let $Q$ be the quiver glued by $Q',Q''$ as the following picture shows:
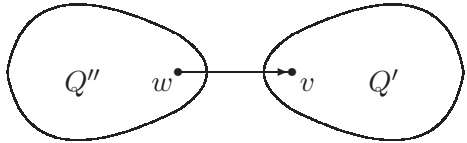
\setlength{\unitlength}{0.5mm}
\begin{center}
\begin{picture}(120,60)(0,-20)
\qbezier(-5,0)(0,30)(40,10)
\qbezier(-5,0)(0,-30)(40,-10)
\qbezier(40,10)(55,0)(40,-10)
\put(40,0){\circle*{2}}

\qbezier(70,10)(55,0)(70,-10)
\qbezier(70,10)(110,30)(115,0)
\qbezier(70,-10)(110,-30)(115,0)
\put(70,0){\circle*{2}}
\put(40,0){\vector(1,0){29}}

\put(33,-5){$w$}
\put(72,-5){$v$}
\put(10,-5){$Q''$}
\put(90,-5){$Q'$}
\end{picture}
\end{center}

Let $\Lambda=KQ/I$, where $I$ is the ideal generated by $I'$ and $I''$.
Then $\Lambda$ is the upper matrix algebra $\left(\begin{array}{cc} B & M\\ 0&C \end{array}\right)$, where $M$ is projective as $B$($C$)-module respectively.
If $B$ and $C$ are Gorenstein algebras, then $\Lambda$ is a Gorenstein algebra.
\begin{lemma}\label{lemma recollement splits}
Keep the notations as above. Then $$\underline{\Gproj}(\Lambda)\simeq \underline{\Gproj}(B)\coprod\underline{\Gproj}(C).$$
\end{lemma}
\begin{proof}
Recall that any module in $\mod\Lambda$ can be identified with a triple $\left( \begin{array}{cc} X\\Y \end{array} \right)_\phi$, where $X\in \mod B$, $Y\in\mod C$, and $\phi:M\otimes_CY\rightarrow X$ is an $B$-map.

By Theorem 2.2 in \cite{Zh1}, we get that $\left( \begin{array}{cc} X\\Y \end{array} \right)_\phi$ is a Gorenstein projective $\Lambda$-module if and only if
$\phi$ is monic, $X$, $M\otimes Y$ and $\Coker\phi$ are Gorenstein projective $B$-modules, and $Y$ is a Gorenstein projective $C$-module.

By Theorem 3.3 in \cite{Zh1}, we get a recollement
\[\xymatrix{
\underline{\Gproj}(B) \ar[r]^{i_*}& \underline{\Gproj}(\Lambda)
\ar@<2ex>[l]^{i^!}\ar@<-3ex>[l]_{i^*}\ar[r]^{j^*}& \underline{\Gproj}(C)
\ar@<2ex>[l]^{j_*}\ar@<-3ex>[l]_{j_!}  }\]
where $i^*$ is given by $\left(\begin{array}{cc} X  \\ Y   \end{array} \right)_{\phi }\mapsto \Coker(\phi )$;
$i_*$ is given by $X \mapsto \left(\begin{array}{cc} X  \\ 0  \end{array} \right)$; $i^!$ is given by $\left(\begin{array}{cc}X  \\ Y   \end{array} \right)_{\phi }\mapsto X $;
$j_!$ is given by $Y  \mapsto\left(\begin{array}{cc} M\otimes Y\\ Y   \end{array} \right)_{\Id}$;
$j^*$ is given by $\left(\begin{array}{cc} X  \\ Y   \end{array} \right)_{\phi }\mapsto Y $;
$j_*$ is given by $ Y \mapsto \left(\begin{array}{cc} P_Y \\ Y   \end{array} \right)_{\sigma}$, where $P_Y$ is a projective object in $\mod C$ such that there is an exact sequence $$0\rightarrow M\otimes Y\xrightarrow{\sigma} P_Y\rightarrow\Coker\sigma\rightarrow0$$ with $\Coker\sigma\in \Gproj(B)$.

So for any Gorenstein projective $\Lambda$-module $\left( \begin{array}{cc} X\\Y \end{array} \right)_\phi$, there exists a triangle in $\underline{\Gproj}(\Lambda)$:
\begin{equation}\label{equation triangle}
\left( \begin{array}{cc} M\otimes Y\\Y \end{array} \right)_{\Id} \rightarrow  \left( \begin{array}{cc} X\\Y \end{array} \right)_\phi \rightarrow \left( \begin{array}{cc} \Coker \phi\\0 \end{array} \right)_0\xrightarrow{f} \left( \begin{array}{cc} M\otimes Z\\Z \end{array} \right)_{\Id},
\end{equation}
where $Z$ is a Gorenstein projective $B$-module such that there is an exact sequence
$0\rightarrow Y \rightarrow Q_Y \rightarrow Z\rightarrow0$ with $Q_Y$ is a projective $C$-module, $Z\in \Gproj(B)$.
For any $C$-module $Z$, $M\otimes Z\cong\oplus_{\dim e_wZ}P_v$ is projective, where $e_w$ is the idempotent corresponding to the vertex $w$ and  $P_v$ is the indecomposable projective corresponding to the vertex $v$. Then $f$ factors through a projective module $\left( \begin{array}{cc} M\otimes Z\\0 \end{array} \right)$, so $f=0$ in $\underline{\Gproj}(\Lambda)$.
Thus the triangle (\ref{equation triangle}) splits. Similarly, we can get that $\Hom_{\underline{\Gproj}(\Lambda)}( \Im i_*,\Im j_!)=0$. We also get that
$\Hom_{\underline{\Gproj}(\Lambda)}( \Im j_!,\Im i_*)=0$ since the adjoint property and $j^*i_*\simeq0$. By all of above, we get that
$\underline{\Gproj}(\Lambda)\simeq \Im i_*\coprod\Im j_!$, and then $$\underline{\Gproj}(\Lambda)\simeq \underline{\Gproj}(B)\coprod\underline{\Gproj}(C),$$
since $i_*,j_!$ are fully embeddings.
\end{proof}

\begin{corollary}\cite{Ka}\label{corollary CM module of cluster-tilted algebra of type A}
Let $\Lambda=KQ/I$ be a cluster-tilted algebra of type $A$.  Then
$$\underline{\Gproj}(\Lambda)\simeq \coprod_{t(Q)} \underline{\mod} S_3,$$

\end{corollary}
\begin{proof}
By Proposition \ref{proposition rooted quiver of type A}, we can assume that $Q$ is in the standard form.
By Theorem 4.1 in \cite{Chen1}, we get that
$\underline{\Gproj}(\Lambda)$ is triangulated equivalent to
$\underline{\Gproj}(B)$, where $B$ is the cluster-tilted algebra of type A as follows
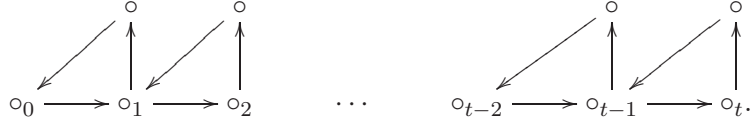
\[\xymatrix {&\circ \ar[dl]& \circ\ar[dl]& && \circ \ar[dl] &\circ \ar[dl] \\
\circ_{0}\ar[r] &\circ_1 \ar[r] \ar[u]& \circ_2\ar[u]&\cdots& \circ_{t-2}\ar[r] &\circ_{t-1} \ar[u]\ar[r] &\circ_t. \ar[u]    }\]
Using Theorem 4.1 in \cite{Chen1}, $\underline{\Gproj}(B)$ is triangulated equivalent to $\underline{\Gproj}(C)$, where $C$ is the cluster-tilted algebra of type A as follows:
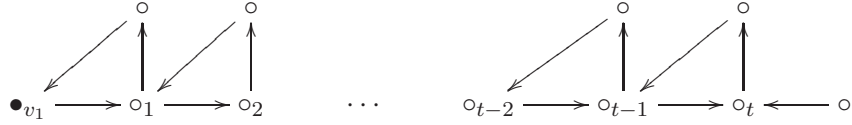
\[\xymatrix {&\circ \ar[dl]& \circ\ar[dl]& && \circ \ar[dl] &\circ \ar[dl] &\\
\bullet_{v_1}\ar[r] &\circ_1 \ar[r] \ar[u]& \circ_2\ar[u]&\cdots& \circ_{t-2}\ar[r] &\circ_{t-1} \ar[u]\ar[r] &\circ_t \ar[u] &\circ\ar[l]  . }\]

From Table 3 or Lemma 3.3 in \cite{BHL2}, we get that the mutation at the vertex $t$ is a good mutation. Applying the mutation at $t$, we get that $C$ is derived equivalent to the cluster-tilted algebra $\Gamma$ of type A
as follows:
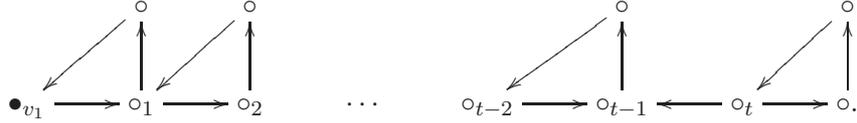
\[\xymatrix {&\circ \ar[dl]& \circ\ar[dl]& && \circ \ar[dl] & & \circ \ar[dl]\\
\bullet_{v_1}\ar[r] &\circ_1 \ar[r] \ar[u]& \circ_2\ar[u]&\cdots& \circ_{t-2}\ar[r] &\circ_{t-1} \ar[u] &\circ_t \ar[l]\ar[r] &\circ. \ar[u]   }\]

Applying Lemma \ref{lemma recollement splits} and inductively, we get that $\underline{\Gproj}(\Gamma)\simeq \coprod_{t(Q)}S_3$. So $$\underline{\Gproj}(\Lambda)\simeq \coprod_{t(Q)}\underline{\mod} S_3.$$
\end{proof}

\begin{theorem}\label{theorem rooted quiver to type A}
Let $Q$ be the gluing of $k$ rooted quivers of type $A$, say
$(Q_1,v_1)$, $(Q_2,v_2),\dots,(Q_k,v_k)$, to $Q_0$ at the distinct vertices $c_1,\dots,c_k$. If $KQ_0/I_0$ is a Gorenstein algebra, then $KQ/I$ is also a Gorenstein algebra and
$$\underline{\Gproj} (KQ/I)\simeq \underline{\Gproj} (KQ_0/I_0)\coprod_{t(Q_1)+t(Q_2)+\cdots+t(Q_k)} \underline{\mod} (S_3).$$
\end{theorem}
\begin{proof}
By induction on $k$, we only need prove that for the case $k=1$.

If $(Q_1,v_1)$ satisfies $s(Q)=3t(Q)$, then it is as the following diagram shows
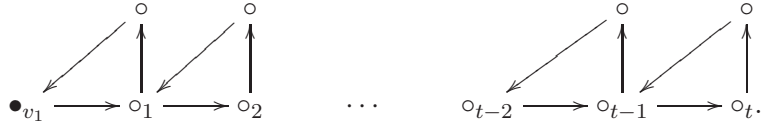
\[\xymatrix {&\circ \ar[dl]& \circ\ar[dl]& && \circ \ar[dl] &\circ \ar[dl] \\
\bullet_{v_1}\ar[r] &\circ_1 \ar[r] \ar[u]& \circ_2\ar[u]&\cdots& \circ_{t-2}\ar[r] &\circ_{t-1} \ar[u]\ar[r] &\circ_t .\ar[u]    }\]

Let $(Q_1',v_1)$ be the rooted quiver
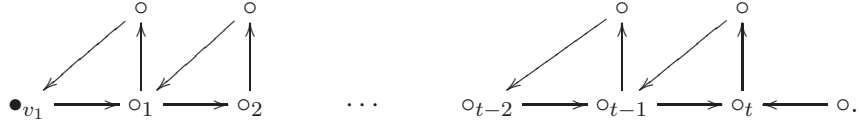
\[\xymatrix {&\circ \ar[dl]& \circ\ar[dl]& && \circ \ar[dl] &\circ \ar[dl] &\\
\bullet_{v_1}\ar[r] &\circ_1 \ar[r] \ar[u]& \circ_2\ar[u]&\cdots& \circ_{t-2}\ar[r] &\circ_{t-1} \ar[u]\ar[r] &\circ_t \ar[u] &\circ.\ar[l] }\]
Let $Q'$ be the gluing of rooted quiver $(Q_1',v_1)$ to $Q_0$ at $c_1$. Then by Theorem 4.1 in \cite{Chen1}, we get that $D^b_{sg}(KQ'/I')\simeq D^b_{sg}(KQ/I)$.

Let $Q''$ be the gluing quiver of rooted quiver $(Q_1'',v_1)$ to $Q_0$ at $c_1$, where $Q_1''$ is as follows
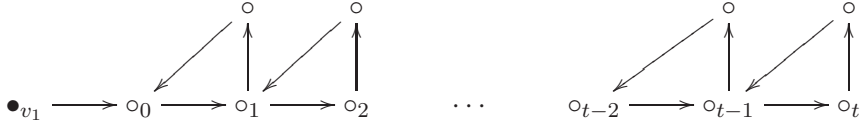
\[\xymatrix {&&\circ \ar[dl]& \circ\ar[dl]& && \circ \ar[dl] &\circ \ar[dl] \\
\bullet_{v_1}\ar[r]&\circ_0 \ar[r] &\circ_1 \ar[r] \ar[u]& \circ_2\ar[u]&\cdots& \circ_{t-2}\ar[r] &\circ_{t-1} \ar[u]\ar[r] &\circ_t \ar[u]    }\]

By Proposition \ref{proposition rooted quiver of type A} and Proposition \ref{proposition Gorenstein algebra under good mutation}, we get that
$KQ'/I'$ is derived equivalent to $KQ''/I''$. From \cite[Theorem 3.3]{Chen1}, we get that $KQ''/I''$ is a Gorenstein algebra and then so is $KQ'/I'$.
Furthermore, $KQ/I$ is also a Gorenstein algebra.

Applying Lemma \ref{lemma recollement splits} and Corollary \ref{corollary CM module of cluster-tilted algebra of type A}, we get that $$\underline{\Gproj } (KQ''/I'')\simeq \underline{\Gproj} (KQ_0/I_0)\coprod_{t} \underline{\mod }(S_3).$$

For the other case $s(Q_1)>3t(Q_1)$, also by Proposition \ref{proposition rooted quiver of type A}, we can assume that $(Q_1,v_1)$ is the rooted quiver as the following diagram shows:
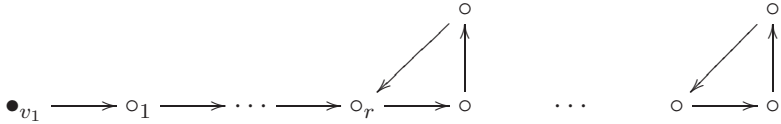
\[\xymatrix {&&&  &\circ \ar[dl]  &&&  \circ \ar[dl]  &\\
\bullet_{v_1}\ar[r] &\circ_1 \ar[r]& \cdots \ar[r] &\circ_r \ar[r] &\circ\ar[u] & \cdots& \circ \ar[r] & \circ \ar[u]    }\]
where $r\geq1$ since $s(Q_1)>3t(Q_1)$.
Then $KQ/I$ is Gorenstein follows from \cite[Theorem 3.3]{Chen1} immediately. By Lemma \ref{lemma recollement splits} and Corollary \ref{corollary CM module of cluster-tilted algebra of type A}, it is easy to see that
$$\underline{\Gproj} (KQ/I)\simeq \underline{\Gproj }(KQ_0/I_0)\coprod_{t} \underline{\mod} (S_3).$$
The proof is complete.
\end{proof}

\section{application: cluster-tilted algebras of type D}

In order to give the quivers and the relations for the cluster-tilted algebras of type $D$, we first recall the following definition.

\begin{definition}\cite{BHL2}\label{definition of Q(m,...)}
Given integers $m \geq 3, r \geq 0$ and an increasing sequence $1 \leq i_1 < i_2 < \cdots < i_r \leq m,$ we define the following quiver $Q(m,\{i_1,\dots,i_r\})$ with relations.

  (a) $Q(m,\{i_1,\dots,i_r\})$ has $m + r$ vertices, labeled $1,2,...,m$ together with $c_1,c_2,\dots,c_r,$ and its arrows are
  $$\{i\rightarrow (i+1)\}_{1\leq i\leq m}\cup \{c_j\rightarrow i_j, (i_j+1)\rightarrow c_j\}_{1\leq j\leq r},$$
\noindent where $i + 1$ is considered modulo $m$.

The full subquiver on the vertices $1,2,\cdots,m$ is thus an oriented cycle of length $m$, called the central cycle, and for every $1 \leq j \leq r$, the full subquiver on the vertices $i_j,i_j + 1,c_j$ is an oriented 3-cycle, called a spike.

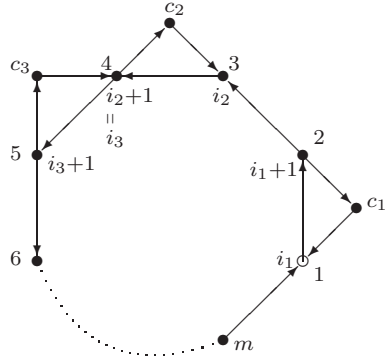
\setlength{\unitlength}{0.7mm}
\begin{center}
\begin{picture}(60,70)(0,-20)
\put(60,0){\circle{2}}
\put(60,0){\vector(0,1){19}}
\put(62,-3){${}_1$}
\put(55,0){${}_{i_1}$}

\put(60,20){\circle*{2}}
\put(60,20){\vector(1,-1){9}}
\put(60,20){\vector(-1,1){14}}
\put(62,23){${}_{2}$}
\put(50,17){${}_{i_1+1}$}

\put(70,10){\circle*{2}}
\put(70,10){\vector(-1,-1){9}}
\put(72,11){${}_{c_1}$}

\put(45,35){\circle*{2}}
\put(45,35){\vector(-1,0){19}}
\put(46,37){${}_{3}$}
\put(43,31){${}_{i_2}$}

\put(25,35){\circle*{2}}
\put(25,35){\vector(1,1){9}}
\put(25,35){\vector(-1,-1){14}}
\put(22,37){${}_{4}$}
\put(23,31){${}_{i_2+1}$}
\put(23,27){${}_{\shortparallel}$}
\put(23,23){${}_{i_3}$}

\put(35,45){\circle*{2}}
\put(35,45){\vector(1,-1){9}}
\put(34,48){${}_{c_2}$}

\put(10,20){\circle*{2}}
\put(10,20){\vector(0,-1){19}}
\put(10,20){\vector(0,1){14}}
\put(5,20){${}_{5}$}
\put(12,18){${}_{i_3+1}$}

\put(10,35){\circle*{2}}
\put(10,35){\vector(1,0){14}}
\put(5,37){${}_{c_3}$}

\put(10,0){\circle*{2}}
\put(5,0){${}_{6}$}

\put(45,-15){\circle*{2}}
\put(45,-15){\vector(1,1){14}}
\put(47,-16){${}_{m}$}

%\put(70,10){\circle*{2}}

%\put(35,45){\circle*{2}}

%\put(10,35){\circle*{2}}
%\put(0,0){\draw0,0,{width(0.2),dotted(6,7),}}
\qbezier[25](10,0)(20,-25)(45,-15)

%\put(10,0){\draw 0,0,{width(.2),dotted(6,7),trace(pi,3*pi,0,cos(t),sin(t))}}

\end{picture}
\end{center}

(b) The relations on $Q(m,\{i_1,\dots,i_r\})$ are as follows:
\begin{itemize}
\item The paths $i_j,i_j + 1,c_j$ and $c_j,i_j,i_j + 1$ are zero for all $1 \leq j \leq r;$
\item For any $1 \leq j \leq r,$ the path $i_j + 1,c_j,i_j$ equals the path $i_j + 1,\dots,1,\dots,i_j$ of length $m-1$ along the central cycle;
\item For any $i\notin  \{i_1,\dots,i_r\}$, the path $i + 1,\dots,1,\dots,i$ of length $m - 1$ along the central cycle is zero.
\end{itemize}
We denote the quotient algebra by $KQ(m,\{i_1,\dots,i_r\})/I$.
\end{definition}

Note that by \cite{GP,V} or seeing Theorem \ref{theorem cluster-tilted algebra of type D} in the following directly, we know that $KQ(m,\{i_1,\dots,i_r\})/I$ is a cluster-tilted algebra of type $D$.

\begin{definition}\label{definition of d_Q}
Let  $Q$ be the quiver of $Q(m,\{i_1,...,i_r\})$, define the distances $d_1,d_2,\dots,d_r$ by
$$d_1=i_2-i_1, d_2=i_3-i_2,\cdots, d_r=i_1+m-i_r,$$
and define  $d_Q$ be the number of the distances where distance is one. %i.e. $d_Q+1$ is the number of spikes of $Q$ which has vertex in common with some other spikes.
\end{definition}

In order to prove the main result of this section, we describe a construction of matrix algebras which is obtained by X-W Chen in \cite{Chen2}, see also \cite[Section 4]{KN}. Let $A$ be a finite dimensional algebra over a field $K$. Let $_AM$ and $N_A$ be a left and right $A$-module, respectively. Then $M\otimes_KN$ becomes an $A\mbox{-}A$-bimodule. Consider an $A\mbox{-}A$-bimodule monomorphism $\phi:M\otimes_K N\rightarrow A$ such that $\Im\phi$ vanishes both on $M$ and $N$. Note that $\Im \phi\subseteq A$ is an ideal. The matrix $\Gamma=\left( \begin{array}{cc} A &M\\ N& K\end{array}\right)$ becomes an associative algebra via the following multiplication
$$\left( \begin{array}{cc} a &m\\ n& \lambda\end{array}\right)\left( \begin{array}{cc} a' &m'\\ n'& \lambda'\end{array}\right)=\left( \begin{array}{cc} aa'+\phi(m\otimes n) &am'+\lambda'm\\ na'+\lambda n'& \lambda\lambda'\end{array}\right).$$
\begin{proposition}\cite{Chen2}\label{proposition homological epimorphism induces singularity equivalences}
Keep the notation and assmption as above. Then there is a triangle equivalence $D^b_{sg}(\Gamma)\simeq D^b_{sg}(A/\Im\phi)$.
\end{proposition}

Then we have the following lemmas which are crucial to prove the main result of this section.

\begin{lemma}\label{lemma singularity category of cluster-tilted algebra of D}
Let $Q$ be the quiver of $Q(m,\{i_1,...,i_r\})$. If the distance $d_j>1$ for any $1\leq j\leq r$, then
$$\underline{\Gproj} (KQ/I)\simeq \underline{\mod }(S_{m}),$$
where $I$ is the ideal generated by the relations defined in Definition \ref{definition of Q(m,...)}.
\end{lemma}
\begin{proof}
We prove the statement by induction on $r$.

Firstly, when $r=0$, we have nothing to prove.

Secondly, when $r>0$, because each distance $i_l-i_{l-1}>1$ for any $l$ modulo $m$, without losing generalization, we can assume that $Q$ is as the following diagram shows.

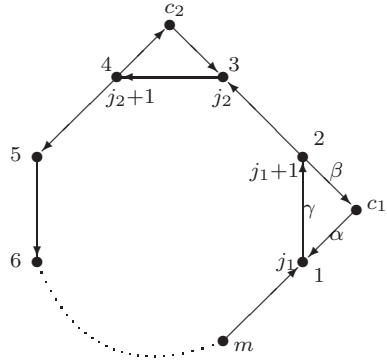
\setlength{\unitlength}{0.7mm}
\begin{center}
\begin{picture}(60,70)(0,-20)
\put(60,0){\circle*{2}}
\put(60,0){\vector(0,1){19}}
\put(62,-3){${}_1$}
\put(55,0){${}_{j_1}$}

\put(60,20){\circle*{2}}
\put(60,20){\vector(1,-1){9}}
\put(60,20){\vector(-1,1){14}}
\put(62,23){${}_{2}$}
\put(50,17){${}_{j_1+1}$}

\put(70,10){\circle*{2}}
\put(70,10){\vector(-1,-1){9}}
\put(72,11){${}_{c_1}$}

\put(45,35){\circle*{2}}
\put(45,35){\vector(-1,0){19}}
\put(46,37){${}_{3}$}
\put(43,31){${}_{j_2}$}

\put(25,35){\circle*{2}}
\put(25,35){\vector(1,1){9}}
\put(25,35){\vector(-1,-1){14}}
\put(22,37){${}_{4}$}
\put(23,31){${}_{j_2+1}$}

\put(35,45){\circle*{2}}
\put(35,45){\vector(1,-1){9}}
\put(34,48){${}_{c_2}$}

\put(10,20){\circle*{2}}
\put(10,20){\vector(0,-1){19}}

\put(5,20){${}_{5}$}

\put(10,0){\circle*{2}}
\put(5,0){${}_{6}$}

\put(45,-15){\circle*{2}}
\put(45,-15){\vector(1,1){14}}
\put(47,-16){${}_{m}$}

\put(65,17){${}_{\beta}$}
\put(65,5){${}_{\alpha}$}
\put(60,10){${}_{\gamma}$}
%\put(70,10){\circle*{2}}

%\put(35,45){\circle*{2}}

%\put(10,35){\circle*{2}}
%\put(0,0){\draw0,0,{width(0.2),dotted(6,7),}}
\qbezier[25](10,0)(20,-25)(45,-15)

%\put(10,0){\draw 0,0,{width(.2),dotted(6,7),trace(pi,3*pi,0,cos(t),sin(t))}}

\end{picture}
\end{center}

Denote by $U=KQ/I$, and let $e_{c_1}$ be the idempotent corresponding to the vertex $c_1$. Let $V$ be the quotient algebra of $U$ module the ideal generated by $e_{c_1}$, namely $Ue_{c_1}U$. The quiver of $V$ is obtained from $Q$ by removing the vertex $c_1$ and the adjacent arrows $\alpha,\beta$.

Let $A=(1-e_{c_1})U(1-e_{c_1})$. Then $K\alpha$ and $K\beta$ are naturally a left and right $A$-module, respectively, since $\beta\gamma=0$ and $\gamma\alpha=0$ in $A$.

We identify $U$ with $\left(\begin{array}{cc} A &K\alpha \\ K\beta &K  \end{array} \right)$, where the $K$ in the southeast corner is identified with $e_{c_1}U e_{c_1}$, and $A=(1-e_{c_1})U(1-e_{c_1})$. The morphism $\phi:K\alpha\otimes_K K\beta\rightarrow A$ maps $\alpha\otimes\beta$ to $\alpha\beta$, so $\phi$ is a monomorphism.
Furthermore, it is easy to see that $A/\Im\phi=V$.
Then Proposition \ref{proposition homological epimorphism induces singularity equivalences} yields a triangulated equivalence $D^b_{sg}(U)\simeq D^b_{sg}(V)$.

Inductively, we can get that $\underline{\Gproj}(U)\simeq D^b_{sg}(U)\simeq D^b_{sg}(S_{m})\simeq\underline{\mod}S_{m}$, and then $\underline{\Gproj}(KQ/I)\simeq \underline{\Gproj}(U)\simeq  \underline{\mod}S_{m}.$
\end{proof}

\begin{lemma}\label{the other lemma singularity category of cluster-tilted algebra of D}
Let $Q$ be the quiver of $Q(m,\{i_1,...,i_r\})$. If $r<m$ and there exists a distance $d_{j_0}=1$ for some $1\leq j_0\leq r$, then
$$\underline{\Gproj} (KQ/I)\simeq \underline{\mod }(S_{m+d_Q}),$$
where $I$ is the ideal generated by the relations defined in Definition \ref{definition of Q(m,...)}.
\end{lemma}
\begin{proof}
Without losing generalization, we assume that $j_0=1$ and $d_r>1$ since $r<m$, then the quiver $Q$
is as the following diagram shows.
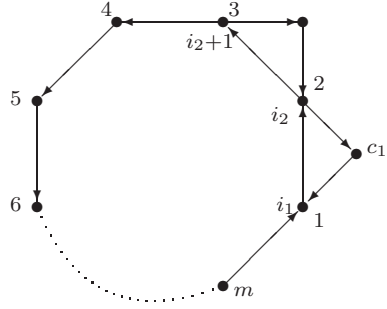
\setlength{\unitlength}{0.7mm}
\begin{center}
\begin{picture}(60,70)(0,-20)
\put(60,0){\circle*{2}}
\put(60,0){\vector(0,1){19}}
\put(62,-3){${}_1$}
\put(55,0){${}_{i_1}$}

\put(60,20){\circle*{2}}
\put(60,20){\vector(1,-1){9}}
\put(60,20){\vector(-1,1){14}}
\put(62,23){${}_{2}$}
\put(54,17){${}_{i_2}$}

\put(45,35){\vector(1,0){14}}
\put(60,35){\circle*{2}}
\put(60,35){\vector(0,-1){14}}
\put(70,10){\circle*{2}}
\put(70,10){\vector(-1,-1){9}}
\put(72,11){${}_{c_1}$}

\put(45,35){\circle*{2}}
\put(45,35){\vector(-1,0){19}}
\put(46,37){${}_{3}$}
\put(38,31){${}_{i_2+1}$}

\put(25,35){\circle*{2}}

\put(25,35){\vector(-1,-1){14}}
\put(22,37){${}_{4}$}

\put(10,20){\circle*{2}}
\put(10,20){\vector(0,-1){19}}

\put(5,20){${}_{5}$}

\put(10,0){\circle*{2}}
\put(5,0){${}_{6}$}

\put(45,-15){\circle*{2}}
\put(45,-15){\vector(1,1){14}}
\put(47,-16){${}_{m}$}

%\put(70,10){\circle*{2}}

%\put(35,45){\circle*{2}}

%\put(10,35){\circle*{2}}
%\put(0,0){\draw0,0,{width(0.2),dotted(6,7),}}
\qbezier[25](10,0)(20,-25)(45,-15)

%\put(10,0){\draw 0,0,{width(.2),dotted(6,7),trace(pi,3*pi,0,cos(t),sin(t))}}

\end{picture}
\end{center}

By the good mutation of cluster-tilted algebras of type $D$ defined in the Table 7 IV.1b in \cite{BHL1}, we know that
the mutation at $c_1$ is good. After applying the mutation $\mu_{c_1}$,
we get that $KQ/I$ is derived equivalent to $KQ'/I'$ where the quiver of $Q'$ is $Q(m+1,\{j_2,...,j_{r}\})$ as the following diagram shows.
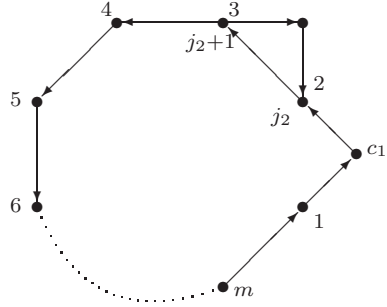
\setlength{\unitlength}{0.7mm}
\begin{center}
\begin{picture}(60,70)(0,-20)
\put(60,0){\circle*{2}}

\put(62,-3){${}_1$}

\put(60,20){\circle*{2}}

\put(60,20){\vector(-1,1){14}}
\put(62,23){${}_{2}$}
\put(54,17){${}_{j_2}$}

\put(45,35){\vector(1,0){14}}
\put(60,35){\circle*{2}}
\put(60,35){\vector(0,-1){14}}
\put(70,10){\circle*{2}}
\put(60,0){\vector(1,1){9}}
\put(70,10){\vector(-1,1){9}}
\put(72,11){${}_{c_1}$}

\put(45,35){\circle*{2}}
\put(45,35){\vector(-1,0){19}}
\put(46,37){${}_{3}$}
\put(38,31){${}_{j_2+1}$}

\put(25,35){\circle*{2}}

\put(25,35){\vector(-1,-1){14}}
\put(22,37){${}_{4}$}

\put(10,20){\circle*{2}}
\put(10,20){\vector(0,-1){19}}

\put(5,20){${}_{5}$}

\put(10,0){\circle*{2}}
\put(5,0){${}_{6}$}

\put(45,-15){\circle*{2}}
\put(45,-15){\vector(1,1){14}}
\put(47,-16){${}_{m}$}

%\put(70,10){\circle*{2}}

%\put(35,45){\circle*{2}}

%\put(10,35){\circle*{2}}
%\put(0,0){\draw0,0,{width(0.2),dotted(6,7),}}
\qbezier[25](10,0)(20,-25)(45,-15)

%\put(10,0){\draw 0,0,{width(.2),dotted(6,7),trace(pi,3*pi,0,cos(t),sin(t))}}

\end{picture}
\end{center}

It is easy to see that $d_{Q'}=d_Q-1$. If $d_Q=1$, then each distance $d_j>1$ for $Q'$, then
$$\underline{\Gproj} (KQ/I)\simeq\underline{\Gproj} (KQ'/I')\simeq\underline{\mod }(S_{m+1}),$$
by Lemma \ref{lemma singularity category of cluster-tilted algebra of D}. Otherwise, if $d_Q>1$, then there exists a distance $d_{j'}=1$ for some $2\leq j'\leq r$. We can apply the good mutations of cluster-tilted algebras of type $D$
same to the above recursively and get that
$$D^b(KQ'/I')\simeq D^b(KQ''/I''),$$
where $Q''$ is the quiver $Q(m+d_Q,\{j'_1,...,j'_{r'}\})$
with each distance $d_i>1$ for $1\leq i\leq r'$.
Then Lemma \ref{lemma singularity category of cluster-tilted algebra of D} yields that
$$\underline{\Gproj} (KQ/I)\simeq\underline{\Gproj} (KQ'/I')\simeq \underline{\Gproj} (KQ''/I'')\simeq\underline{\mod }(S_{m+d_Q}).$$

\end{proof}
\begin{lemma}\label{the final lemma singularity category of cluster-tilted algebra of D}
Let $Q$ be the quiver of $Q(m,\{1,...,m\})$. Then
$$\underline{\Gproj} (KQ/I)\simeq \underline{\mod }(S_{2m}),$$
where $I$ is the ideal generated by the relations defined in Definition \ref{definition of Q(m,...)}.
\end{lemma}

\begin{proof}
The quiver of $Q$ is as the following diagram shows.
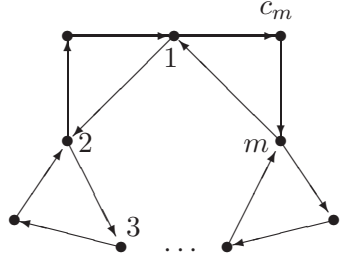
\setlength{\unitlength}{0.7mm}
\begin{center}

\begin{picture}(100,60)
\put(20,50){\circle*{2}}
\put(40,50){\circle*{2}}
\put(60,50){\circle*{2}}

\put(56,54){$c_m$}

\put(20,50){\vector(1,0){19}}
\put(40,50){\vector(1,0){19}}

\put(20,30){\circle*{2}}
\put(60,30){\circle*{2}}
\put(30,10){\circle*{2}}
\put(50,10){\circle*{2}}
\put(10,15){\circle*{2}}
\put(70,15){\circle*{2}}
\put(20,30){\vector(0,1){19}}
\put(60,50){\vector(0,-1){19}}
\put(40,50){\vector(-1,-1){19}}
\put(60,30){\vector(-1,1){19}}
\put(20,30){\vector(1,-2){9}}

\put(30,10){\vector(-4,1){19}}
\put(10,15){\vector(2,3){9}}
\put(60,30){\vector(2,-3){9}}

\put(50,10){\vector(1,2){9}}
\put(70,15){\vector(-4,-1){19}}

\put(38,44){$1$}
\put(53,28){$m$}
\put(22,28){$2$}
\put(31,12){$3$}
\put(38,8){$\cdots$}
\end{picture}
\end{center}

Furthermore, let $KQ'/I'$ be the cluster-tilted algebra of type $D$ with the quiver $Q'$ as follows. Note that $KQ'/I'$ is the one-point extension algebra of $KQ/I$ by the indecomposable projective module $P_{c_m}$ corresponding to the vertex $c_m$. Then \cite[Theorem 4.1]{Chen1} implies that $D^b_{sg}(KQ/I)\simeq D^b_{sg}(KQ'/I')$.
\setlength{\unitlength}{0.7mm}
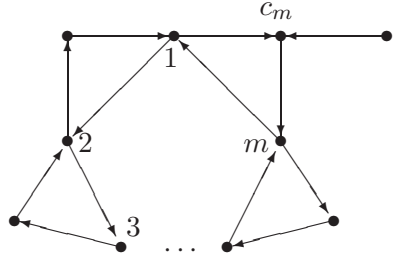
\begin{center}

\begin{picture}(100,60)
\put(20,50){\circle*{2}}
\put(40,50){\circle*{2}}
\put(60,50){\circle*{2}}
\put(80,50){\circle*{2}}
\put(56,54){$c_m$}

\put(20,50){\vector(1,0){19}}
\put(40,50){\vector(1,0){19}}
\put(80,50){\vector(-1,0){19}}

\put(20,30){\circle*{2}}
\put(60,30){\circle*{2}}
\put(30,10){\circle*{2}}
\put(50,10){\circle*{2}}
\put(10,15){\circle*{2}}
\put(70,15){\circle*{2}}
\put(20,30){\vector(0,1){19}}
\put(60,50){\vector(0,-1){19}}
\put(40,50){\vector(-1,-1){19}}
\put(60,30){\vector(-1,1){19}}
\put(20,30){\vector(1,-2){9}}

\put(30,10){\vector(-4,1){19}}
\put(10,15){\vector(2,3){9}}

\put(60,30){\vector(2,-3){9}}

\put(50,10){\vector(1,2){9}}
\put(70,15){\vector(-4,-1){19}}
\put(38,44){$1$}
\put(53,28){$m$}
\put(22,28){$2$}
\put(31,12){$3$}
\put(38,8){$\cdots$}

\end{picture}
\end{center}

By the good mutations of cluster-tilted algebras of type $D$ defined in the Table 7 IV.2b in \cite{BHL1}, we get that the mutation at the vertex $c_m$ is good. After applying the mutation $\mu_{c_m}$, we get the cluster-tilted algebra $KQ''/I''$ of type $D$ with the quiver $Q''=Q(m+1,\{1,\dots,m\})$ as the following diagram shows.

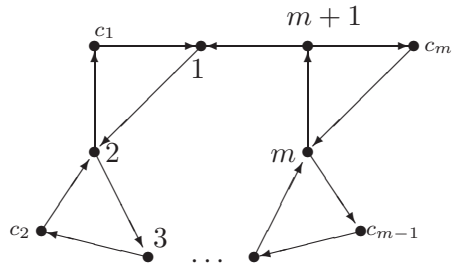
\setlength{\unitlength}{0.7mm}
\begin{center}

\begin{picture}(100,60)
\put(20,50){\circle*{2}}
\put(40,50){\circle*{2}}
\put(60,50){\circle*{2}}
\put(80,50){\circle*{2}}
\put(56,54){$m+1$}

\put(20,50){\vector(1,0){19}}
\put(60,50){\vector(-1,0){19}}
\put(60,50){\vector(1,0){19}}
\put(80,50){\vector(-1,-1){19}}

\put(20,50){${}^{c_1}$}

\put(4,15){${ }_{c_2}$}

\put(20,30){\circle*{2}}
\put(60,30){\circle*{2}}
\put(30,10){\circle*{2}}
\put(50,10){\circle*{2}}
\put(10,15){\circle*{2}}
\put(70,15){\circle*{2}}
\put(20,30){\vector(0,1){19}}
\put(60,30){\vector(0,1){19}}
\put(40,50){\vector(-1,-1){19}}

\put(20,30){\vector(1,-2){9}}

\put(30,10){\vector(-4,1){19}}
\put(10,15){\vector(2,3){9}}

\put(60,30){\vector(2,-3){9}}

\put(50,10){\vector(1,2){9}}
\put(70,15){\vector(-4,-1){19}}
\put(38,44){$1$}
\put(53,28){$m$}
\put(22,28){$2$}
\put(31,12){$3$}
\put(38,8){$\cdots$}
\put(71,15){${}_{c_{m-1}}$}
\put(82,50){${}_{c_m}$}
\end{picture}
\end{center}

It follows from Lemma \ref{the other lemma singularity category of cluster-tilted algebra of D} that
$$D^b_{sg}(KQ''/I'')\simeq \underline{\mod}S_{m+1+d_{Q''}}=\underline{\mod}S_{2m}$$
since $d_{Q''}=m-1$. Thus $$D^b_{sg}(KQ/I)\simeq D^b_{sg}(KQ'/I')\simeq D^b_{sg}(KQ''/I'')\simeq \underline{\mod}S_{2m}.$$

\end{proof}

\begin{proposition}\label{proposition to type Q(m,{i_1,...,i_r})}
Let $Q$ be the quiver of $Q(m,\{i_1,...,i_r\})$. Then
$$\underline{\Gproj} (KQ/I)\simeq \underline{\mod }(S_{m+d_Q}),$$
where $I$ is the ideal generated by the relations defined in Definition \ref{definition of Q(m,...)}.
\end{proposition}

\begin{proof}
Note that when $r=0$, the result is obvious. Otherwise, we discuss it in the following two cases.

Case 1: $0<r<m$. First, if each distance $d_i>1$ for any $1\leq i\leq r$, then Lemma \ref{lemma singularity category of cluster-tilted algebra of D} yields
$$\underline{\Gproj} (KQ/I)\simeq \underline{\mod }(S_{m})= \underline{\mod }(S_{m+d_Q}),$$
since $d_Q=0$.

Second, if there exists a distance $d_{j}=1$ for some $1\leq j\leq r$, then we get that
$$\underline{\Gproj} (KQ/I)\simeq \underline{\mod }(S_{m+d_Q})$$
by Lemma \ref{the other lemma singularity category of cluster-tilted algebra of D}.

Case 2: $r=m$. Then the quiver $Q$ is $Q(m,\{1,...,m\})$, so it follows from Lemma \ref{the final lemma singularity category of cluster-tilted algebra of D} that
$$ \underline{\Gproj} (KQ/I)\simeq \underline{\mod }(S_{2m})= \underline{\mod }(S_{m+d_Q})$$
since $d_Q=m$.
\end{proof}

In order to get the singularity categories of the cluster-tilted algebra of type $D$, first we describe the structures of their quivers.

\begin{theorem}\cite{GP,V}\label{theorem cluster-tilted algebra of type D}
A cluster-tilted algebra of type $D$ is belong to one of the following four families, which are  defined as gluing of rooted quivers of type $A$ to certain skeleta.

\noindent{\bf Type I.} The gluing of a rooted quiver $Q'$ of type $A$ at the vertex $c$ to the skeleta
as in the following picture where both $a$ and $b$ have an arrow to or from the same vertex $c$:

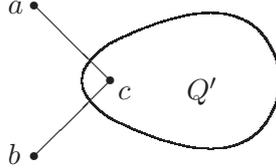
\setlength{\unitlength}{0.5mm}
\begin{center}
\begin{picture}(120,40)(0,-20)

\put(50,20){\circle*{2}}
\put(50,-20){\circle*{2}}
\qbezier(70,10)(55,0)(70,-10)
\qbezier(70,10)(110,30)(115,0)
\qbezier(70,-10)(110,-30)(115,0)
\put(70,0){\circle*{2}}
\put(50,20){\line(1,-1){20}}
\put(50,-20){\line(1,1){20}}

\put(43,18){$a$}
\put(43,-22){$b$}
\put(72,-5){$c$}

\put(90,-5){$Q'$}
\end{picture}
\end{center}

\noindent{\bf Type II.} The gluing of two rooted quivers $Q'$ and $Q''$ of type $A$ at the vertices $c'$ and $c''$ to a skeleton as following:

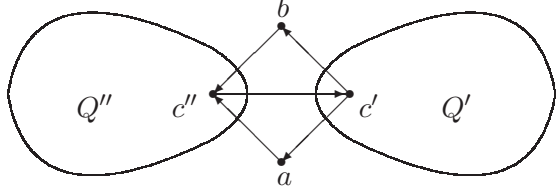
\setlength{\unitlength}{0.6mm}
\begin{center}
\begin{picture}(120,40)(0,-20)
\qbezier(-5,0)(0,30)(40,10)
\qbezier(-5,0)(0,-30)(40,-10)
\qbezier(40,10)(55,0)(40,-10)
\put(40,0){\circle*{2}}

\qbezier(70,10)(55,0)(70,-10)
\qbezier(70,10)(110,30)(115,0)
\qbezier(70,-10)(110,-30)(115,0)
\put(70,0){\circle*{2}}
\put(40,0){\vector(1,0){29}}
\put(55,15){\circle*{2}}
\put(55,-15){\circle*{2}}
\put(55,15){\vector(-1,-1){15}}
\put(55,-15){\vector(-1,1){15}}
\put(70,0){\vector(-1,-1){15}}
\put(70,0){\vector(-1,1){15}}

\put(31,-5){$c''$}
\put(72,-5){$c'$}

\put(54,17){$b$}
\put(54,-20){$a$}
\put(10,-5){$Q''$}
\put(90,-5){$Q'$}
\end{picture}
\end{center}

\noindent{\bf Type III.} The gluing of two rooted quivers $Q'$ and $Q''$ of type $A$ at the vertices $c'$ and $c''$ to a skeleton $S_4$ as following:

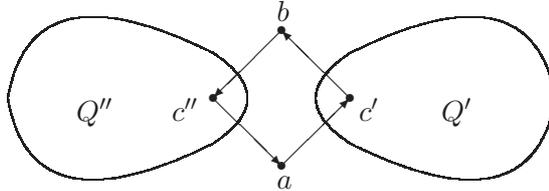
\setlength{\unitlength}{0.6mm}
\begin{center}
\begin{picture}(120,40)(0,-20)
\qbezier(-5,0)(0,30)(40,10)
\qbezier(-5,0)(0,-30)(40,-10)
\qbezier(40,10)(55,0)(40,-10)
\put(40,0){\circle*{2}}

\qbezier(70,10)(55,0)(70,-10)
\qbezier(70,10)(110,30)(115,0)
\qbezier(70,-10)(110,-30)(115,0)
\put(70,0){\circle*{2}}

\put(55,15){\circle*{2}}
\put(55,-15){\circle*{2}}
\put(55,15){\vector(-1,-1){15}}
\put(40,0){\vector(1,-1){15}}
\put(55,-15){\vector(1,1){15}}
\put(70,0){\vector(-1,1){15}}

\put(31,-5){$c''$}
\put(72,-5){$c'$}

\put(54,17){$b$}
\put(54,-20){$a$}
\put(10,-5){$Q''$}
\put(90,-5){$Q'$}
\end{picture}
\end{center}

\noindent{\bf Type IV.}
The gluing of $r \geq 0$ rooted quivers $Q_1,\dots,Q_r$ of type $A$ at the vertices $c_1,\dots,c_r$ to a skeleton $Q(m,\{i_1,...,i_r\})$ as following:

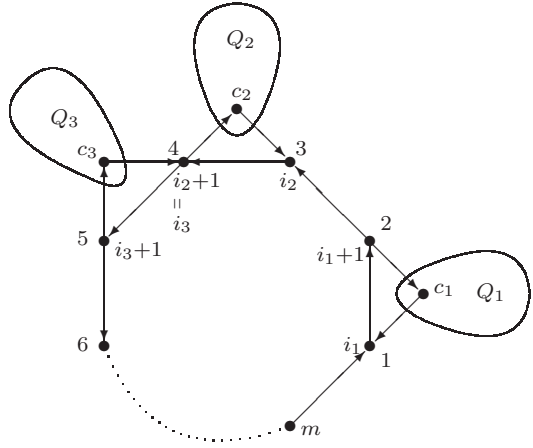
\setlength{\unitlength}{0.7mm}
\begin{center}
\begin{picture}(60,90)(0,-20)
\put(60,0){\circle*{2}}
\put(60,0){\vector(0,1){19}}
\put(62,-3){${}_1$}
\put(55,0){${}_{i_1}$}

\put(60,20){\circle*{2}}
\put(60,20){\vector(1,-1){9}}
\put(60,20){\vector(-1,1){14}}
\put(62,23){${}_{2}$}
\put(50,17){${}_{i_1+1}$}

\put(70,10){\circle*{2}}
\put(70,10){\vector(-1,-1){9}}
\put(72,11){${}_{c_1}$}

\put(45,35){\circle*{2}}
\put(45,35){\vector(-1,0){19}}
\put(46,37){${}_{3}$}
\put(43,31){${}_{i_2}$}

\put(25,35){\circle*{2}}
\put(25,35){\vector(1,1){9}}
\put(25,35){\vector(-1,-1){14}}
\put(22,37){${}_{4}$}
\put(23,31){${}_{i_2+1}$}
\put(23,27){${}_{\shortparallel}$}
\put(23,23){${}_{i_3}$}

\put(35,45){\circle*{2}}
\put(35,45){\vector(1,-1){9}}
\put(34,48){${}_{c_2}$}

\put(10,20){\circle*{2}}
\put(10,20){\vector(0,-1){19}}
\put(10,20){\vector(0,1){14}}
\put(5,20){${}_{5}$}
\put(12,18){${}_{i_3+1}$}

\put(10,35){\circle*{2}}
\put(10,35){\vector(1,0){14}}
\put(5,37){${}_{c_3}$}

\put(10,0){\circle*{2}}
\put(5,0){${}_{6}$}

\put(45,-15){\circle*{2}}
\put(45,-15){\vector(1,1){14}}
\put(47,-16){${}_{m}$}

\qbezier(70,15)(60,10)(70,5)
\qbezier(70,15)(89,23)(90,10)
\qbezier(70,5)(89,-3)(90,10)
\put(80,10){${}_{Q_1}$}
%\put(70,10){\circle*{2}}

\qbezier(30,45)(35,35)(40,45)
\qbezier(30,45)(22,64)(35,65)
\qbezier(40,45)(48,64)(35,65)
\put(33,58){${}_{Q_2}$}
%\put(35,45){\circle*{2}}

\qbezier(7,32)(17,27)(13,38)
\qbezier(7,32)(-14,41)(-5,50)
\qbezier(13,38)(4,59)(-5,50)
\put(0,43){${}_{Q_3}$}
%\put(10,35){\circle*{2}}
%\put(0,0){\draw0,0,{width(0.2),dotted(6,7),}}
\qbezier[25](10,0)(20,-25)(45,-15)

%\put(10,0){\draw 0,0,{width(.2),dotted(6,7),trace(pi,3*pi,0,cos(t),sin(t))}}

\end{picture}
\end{center}

\end{theorem}

Now, we get the main theorem of this section.

\begin{theorem}\label{theorem singularity category of D}
Let $\Lambda$ be a cluster-tilted algebra of type $D$.

(a) If $\Lambda$ is of {\bf Type I}, then $$\underline{\Gproj}(\Lambda)\simeq \coprod_{t(Q')} \underline{\Gproj}(S_3).$$

(b) If $\Lambda$ is of {\bf type II}, then
$$\underline{\Gproj}(\Lambda)\simeq \coprod_{t(Q')+t(Q'')+1} \underline{\Gproj}(S_3).$$

(c) If $\Lambda$ is of {\bf type III}, then $$\underline{\Gproj}(\Lambda)\simeq \coprod_{t(Q')+t(Q'')} \underline{\Gproj}(S_3)\coprod\underline{\Gproj}(S_4).$$

(d) If $\Lambda$ is of {\bf type IV}, then  $$\underline{\Gproj}(\Lambda)\simeq \coprod_{t(Q_1)+t(Q_2)+\dots+t(Q_r)} \underline{\Gproj}(S_3)\coprod\underline{\Gproj}(S_{m+d_Q}),$$
where $d_Q$ is defined for $Q(m,\{i_1,...,i_r\})$ as in Definition \ref{definition of d_Q}.
\end{theorem}
\begin{proof}
(a) If $\Lambda$ is of {\bf Type I}, then Theorem \ref{theorem rooted quiver to type A} yields
$$\underline{\Gproj}(\Lambda)\simeq \underline{\Gproj}(KQ'')\coprod_{t(Q')} \underline{\mod}(S_3),$$
where the underlying graph of $Q''$ is
\[\xymatrix{\bullet ^a\ar@{-}[r] &\bullet^b\ar@{-}[r]&\bullet^c.}\]
It is easy to see that $\underline{\Gproj}(KQ'')=0$, and then $$\underline{\Gproj}(\Lambda)\simeq \coprod_{t(Q')} \underline{\Gproj}(S_3).$$

(b) If $\Lambda$ is of {\bf Type II}, then Theorem \ref{theorem rooted quiver to type A} yields
$$\underline{\Gproj}(\Lambda)\simeq \underline{\Gproj}(KQ'''/I''')\coprod_{t(Q')+t(Q'')} \underline{\mod}(S_3),$$
where $KQ'''/I'''$ is a cluster-tilted algebra of type $D_4$ with the quiver $Q'''$ as the following diagram shows.
\setlength{\unitlength}{0.8mm}
\begin{center}
\begin{picture}(120,40)(0,-20)

\put(40,0){\circle*{1.3}}

\put(70,0){\circle*{1.3}}
\put(40,0){\vector(1,0){29}}
\put(55,15){\circle*{1.3}}
\put(55,-15){\circle*{1.3}}
\put(55,15){\vector(-1,-1){15}}
\put(55,-15){\vector(-1,1){15}}
\put(70,0){\vector(-1,-1){15}}
\put(70,0){\vector(-1,1){15}}

\put(31,-5){$c''$}
\put(72,-5){$c'$}

\put(54,17){$b$}
\put(54,-20){$a$}
\end{picture}
\end{center}
Then Lemma \ref{the other lemma singularity category of cluster-tilted algebra of D} implies that
$$\underline{\Gproj}(KQ'''/I''')\simeq \underline{\mod}S_3.$$
Thus $$\underline{\Gproj}(\Lambda)\simeq \coprod_{t(Q')+t(Q'')+1} \underline{\mod}(S_3),$$

(c) If $\Lambda$ is of {\bf type III}, then the proof is similar to the {\bf type II}, we omit it here.

(d) If $\Lambda$ is of {\bf type IV}, then Theorem \ref{theorem rooted quiver to type A} yields
$$\underline{\Gproj}(\Lambda)\simeq\coprod_{t(Q_1)+t(Q_2)+\dots+t(Q_r)} \underline{\Gproj}(S_3)\coprod\underline{\Gproj}(KQ'/I'),$$
where $KQ'/I'$ is the cluster-tilted algebra of type $D$ with the quiver $Q'$ as the following diagram shows.
\setlength{\unitlength}{0.7mm}
\begin{center}
\begin{picture}(60,70)(0,-20)
\put(60,0){\circle*{2}}
\put(60,0){\vector(0,1){19}}
\put(62,-3){${}_1$}
\put(55,0){${}_{i_1}$}

\put(60,20){\circle*{2}}
\put(60,20){\vector(1,-1){9}}
\put(60,20){\vector(-1,1){14}}
\put(62,23){${}_{2}$}
\put(50,17){${}_{i_1+1}$}

\put(70,10){\circle*{2}}
\put(70,10){\vector(-1,-1){9}}
\put(72,11){${}_{c_1}$}

\put(45,35){\circle*{2}}
\put(45,35){\vector(-1,0){19}}
\put(46,37){${}_{3}$}
\put(43,31){${}_{i_2}$}

\put(25,35){\circle*{2}}
\put(25,35){\vector(1,1){9}}
\put(25,35){\vector(-1,-1){14}}
\put(22,37){${}_{4}$}
\put(23,31){${}_{i_2+1}$}
\put(23,27){${}_{\shortparallel}$}
\put(23,23){${}_{i_3}$}

\put(35,45){\circle*{2}}
\put(35,45){\vector(1,-1){9}}
\put(34,48){${}_{c_2}$}

\put(10,20){\circle*{2}}
\put(10,20){\vector(0,-1){19}}
\put(10,20){\vector(0,1){14}}
\put(5,20){${}_{5}$}
\put(12,18){${}_{i_3+1}$}

\put(10,35){\circle*{2}}
\put(10,35){\vector(1,0){14}}
\put(5,37){${}_{c_3}$}

\put(10,0){\circle*{2}}
\put(5,0){${}_{6}$}

\put(45,-15){\circle*{2}}
\put(45,-15){\vector(1,1){14}}
\put(47,-16){${}_{m}$}
\qbezier[25](10,0)(20,-25)(45,-15)

\end{picture}
\end{center}

Proposition \ref{proposition to type Q(m,{i_1,...,i_r})} shows that $\underline{\Gproj}(KQ'/I')\simeq \underline{\mod}S_{m+d_{Q}}$ since $d_{Q'}=d_Q$. Therefore $$\underline{\Gproj}(\Lambda)\simeq \coprod_{t(Q_1)+t(Q_2)+\dots+t(Q_r)} \underline{\Gproj}(S_3)\coprod\underline{\Gproj}(S_{m+d_Q}).$$

\end{proof}

\section{Cluster-tilted algebras of type E}

Let $A$ be a finite dimensional algebra over a field $K$ and $P_1,\dots,P_n$ be a complete collection of nonisomorphic indecomposable projective $A$-modules. The \emph{Cartan matrix} of $A$ is then the $n\times n$ matrix $C_A$ defined by $(C_A)_{ij}=\dim_K \Hom_A(P_j,P_i)$.

Assume that $C_A$ is invertible over $\Q$, which is satisfied by cluster-tilted algebras of type $E$ \cite{BHL1}. Let $S_A=C_AC_A^{-T}$(here $C_A^{-T}$ denotes the inverse of the transpose of $C_A$), known in the theory of non-symmetric bilinear forms as the \emph{asymmetry} of $C_A$.
\begin{remark}\cite{BHL1}
The matrix $S_A=C_AC_A^{-T}$ is related to the Coxeter transformation which has been widely studied in the case when $A$ has finite global dimension. When $A$ is Gorenstein, the matrix $S_A$ (if it makes sense) is integral since the injective modules have finite projective resolutions. By a result of Keller and Reiten \cite{KR}, this is the case for the cluster-tilted algebras.
\end{remark}
\begin{definition}\cite{BHL1}
Let $A$ be a finite dimensional algebra over a field $K$ and $C_A$ its Cartan matrix. If $C_A$ is invertible over $\Q$, then we associated a polynomial to $A$ as the product of the determinant of the Cartan matrix and the characteristic polynomial of its asymmetry matrix $S_A$, which is called the associated polynomial of $A$.
\end{definition}

In \cite{BHL1}, Bastian, Holm and Ladkani obtained a complete derived equivalence classification for the cluster-tilted algebras of Dynkin type $E$.

\begin{theorem}\cite{BHL1}
Two cluster-tilted algebras of Dynkin type $E$ are derived equivalent if and only if their Cartan matrices represent equivalent bilinear forms over $\Z$. This in turn happens if and only if they have the same associated polynomials.
\end{theorem}

According to these, in order to obtain a complete singularity equivalence classification for the cluster-tilted algebras of Dynkin type $E$, we only need compute the singularity category of a representative in each derived equivalence class, since derived equivalences implies singularity equivalences.

Because the derived equivalent classes parameterized by the associated polynomials \cite{BHL1}, we can get the singularity equivalence classes of Dynkin $E$ in the following tables. For each associated polynomial, we provide a representative, whose arrows which are displayed without orientation can be oriented arbitrarily, without changing the singularity equivalence
class.
\[
\setlength{\tabcolsep}{2pt}
\renewcommand{\arraystretch}{1.2}
\begin{tabular}{|l|l|}
\hline
\multicolumn{2}{|c|}{Singularity equivalence classes of Dynkin type $E_6$}\\
\hline
\multicolumn{1}{|c|}{The quiver and associated polynomial of $A$}  &\multicolumn{1}{|c|}{$\underline{\Gproj}(A)$}\\
\hline
\multicolumn{1}{|c|}{\setlength{\unitlength}{0.7mm}
                     \begin{picture}(0,15)
                     \put(-20,0){\circle*{1.3}}
                     \put(-20,0){\line(1,0){10}}
                     \put(-10,0){\circle*{1.3}}
                     \put(-10,0){\line(1,0){10}}
                     \put(0,0){\circle*{1.3}}
                     \put(0,0){\line(1,0){10}}
                     \put(10,0){\circle*{1.3}}
                     \put(10,0){\line(1,0){10}}
                     \put(20,0){\circle*{1.3}}
                     \put(0,0){\line(0,1){10}}
                     \put(0,10){\circle*{1.3}}
                     \end{picture} }  &\multicolumn{1}{|c|}{$0$}\\
\multicolumn{1}{|c|}{$x^6-x^5+x^3-x+1$} &\\
\hline
\multicolumn{1}{|c|}{\setlength{\unitlength}{0.7mm}
                     \begin{picture}(0,15)
                     \put(-20,0){\circle*{1.3}}
                     \put(-20,0){\vector(1,0){9}}
                     \put(-10,0){\circle*{1.3}}
                     \put(-10,0){\line(1,0){10}}
                     \put(0,0){\circle*{1.3}}
                     \put(0,0){\line(1,0){10}}
                     \put(10,0){\circle*{1.3}}
                     \put(-10,0){\line(0,1){10}}
                     \put(-10,10){\circle*{1.3}}
                     \put(-10,0){\vector(-1,1){9}}
                     \put(-20,10){\circle*{1.3}}
                     \put(-20,10){\vector(0,-1){9}}
                     \end{picture}     }  &\multicolumn{1}{|c|}{$\underline{\mod}(S_3)$}\\
                     \multicolumn{1}{|c|}{$2(x^6-x^4+2x^3-x^2+1)$}  &\\
\hline
\end{tabular}
\]
\[
\setlength{\tabcolsep}{2pt}
\renewcommand{\arraystretch}{1.2}
\begin{tabular}{|l|l|}
\hline
\multicolumn{1}{|c|}{ \setlength{\unitlength}{0.7mm}
                     \begin{picture}(0,25)
                     \put(-20,10){\circle*{1.3}}
                     \put(-20,10){\line(1,0){10}}
                     \put(-10,10){\circle*{1.3}}
                     \put(-10,10){\vector(1,-1){9}}
                     \put(0,0){\circle*{1.3}}
                     \put(0,0){\vector(0,1){20}}
                     \put(0,20){\circle*{1.3}}
                     \put(0,20){\vector(-1,-1){9}}
                     \put(0,20){\vector(1,-1){9}}
                     \put(10,10){\circle*{1.3}}
                     \put(10,10){\vector(-1,-1){9}}
                     \put(10,10){\line(1,0){10}}
                     \put(20,10){\circle*{1.3}}
                     \end{picture}    }  &\multicolumn{1}{|c|}{$\underline{\mod}(S_3)$}\\
\multicolumn{1}{|c|}{$2(x^6-2x^4+4x^3-2x^2+1)$}  &\\
\hline

\multicolumn{1}{|c|}{ \setlength{\unitlength}{0.7mm}
                     \begin{picture}(0,25)
                     \put(-20,10){\circle*{1.3}}
                     \put(-20,10){\vector(1,-1){9}}
                     \put(-10,0){\circle*{1.3}}
                     \put(-10,0){\vector(1,1){9}}
                     \put(0,10){\circle*{1.3}}
                     \put(0,10){\vector(-1,1){9}}
                     \put(-10,20){\circle*{1.3}}
                     \put(-10,20){\vector(-1,-1){9}}
                     \put(-10,20){\line(1,0){10}}
                     \put(0,20){\circle*{1.3}}
                     \put(0,10){\line(1,0){10}}
                     \put(10,10){\circle*{1.3}}
                     \end{picture}    }  &\multicolumn{1}{|c|}{$\underline{\mod}(S_4)$}\\
\multicolumn{1}{|c|}{$3(x^6+x^3+1)$}  &\\
\hline
\multicolumn{1}{|c|}{ \setlength{\unitlength}{0.7mm}
                     \begin{picture}(0,25)
                     \put(-20,15){\circle*{1.3}}
                     \put(-20,15){\vector(0,-1){9}}
                     \put(-20,5){\circle*{1.3}}
                     \put(-20,5){\vector(2,-1){9}}
                     \put(-10,0){\circle*{1.3}}
                     \put(-10,0){\vector(1,1){9}}
                     \put(0,10){\circle*{1.3}}
                     \put(0,10){\vector(-1,1){9}}
                     \put(-10,20){\vector(-2,-1){9}}
                     \put(-10,20){\circle*{1.3}}
                     \put(0,10){\line(1,0){10}}
                     \put(10,10){\circle*{1.3}}
                     \end{picture}  }  &\multicolumn{1}{|c|}{$\underline{\mod}(S_5)$}\\
\multicolumn{1}{|c|}{$4(x^6+x^4+x^2+1)$}  &\\
\hline
\multicolumn{1}{|c|}{\setlength{\unitlength}{0.7mm}
                     \begin{picture}(0,15)
                     \put(-20,0){\circle*{1.3}}
                     \put(-20,0){\vector(1,0){9}}
                     \put(-10,0){\circle*{1.3}}
                     \put(-10,0){\vector(1,0){9}}
                     \put(0,0){\circle*{1.3}}
                     \put(0,0){\vector(0,1){9}}
                     \put(0,10){\circle*{1.3}}
                     \put(0,10){\vector(-1,-1){9}}
                     \put(-10,10){\circle*{1.3}}
                     \put(-10,0){\vector(-1,1){9}}
                     \put(-20,10){\circle*{1.3}}
                     \put(-20,10){\vector(0,-1){9}}
                      \put(-10,0){\line(0,1){10}}
                     \end{picture}   }  &\multicolumn{1}{|c|}{$\underline{\mod}(S_3) \coprod \underline{\mod}(S_3)$}\\
\multicolumn{1}{|c|}{$4(x^6+x^5-x^4+2x^3-x^2+x+1)$}  &\\
\hline
\end{tabular}
\]

\[\setlength{\tabcolsep}{2pt}
\renewcommand{\arraystretch}{1.2}
\begin{tabular}{|l|l|}
\hline
\multicolumn{2}{|c|}{Singularity equivalence classes of Dynkin type $E_7$}\\
\hline
\multicolumn{1}{|c|}{The quiver and associated polynomial of $A$}  &\multicolumn{1}{|c|}{$\underline{\Gproj}(A)$}\\
\hline
\multicolumn{1}{|c|}{ \setlength{\unitlength}{0.7mm}
                     \begin{picture}(0,15)
                     \put(-20,0){\circle*{1.3}}
                     \put(-20,0){\line(1,0){10}}
                     \put(-10,0){\circle*{1.3}}
                     \put(-10,0){\line(1,0){10}}
                     \put(0,0){\circle*{1.3}}
                     \put(0,0){\line(1,0){10}}
                     \put(10,0){\circle*{1.3}}
                     \put(10,0){\line(1,0){10}}
                     \put(20,0){\circle*{1.3}}
                     \put(0,0){\line(0,1){10}}
                     \put(0,10){\circle*{1.3}}
                     \put(20,0){\line(1,0){10}}
                     \put(30,0){\circle*{1.3}}
                     \end{picture} }  &\multicolumn{1}{|c|}{$0$}\\
\multicolumn{1}{|c|}{$x^7-x^6+x^4-x^3+x-1$}  &\\
\hline
\multicolumn{1}{|c|}{ \setlength{\unitlength}{0.7mm}
                     \begin{picture}(0,15)
                     \put(-20,0){\circle*{1.3}}
                     \put(-20,0){\vector(1,0){9}}
                     \put(-10,0){\circle*{1.3}}
                     \put(-10,0){\line(1,0){10}}
                     \put(0,0){\circle*{1.3}}
                     \put(0,0){\line(1,0){10}}
                     \put(10,0){\circle*{1.3}}
                     \put(-10,0){\line(0,1){10}}
                     \put(-10,10){\circle*{1.3}}
                     \put(-10,0){\vector(-1,1){9}}
                     \put(-20,10){\circle*{1.3}}
                     \put(-20,10){\vector(0,-1){9}}
                     \put(10,0){\line(1,0){10}}
                     \put(20,0){\circle*{1.3}}
                     \end{picture}     }  &\multicolumn{1}{|c|}{}\\
\multicolumn{1}{|c|}{$2(x^7-x^5+2x^4-2x^3+x^2-1)$}  &\\
\multicolumn{1}{|c|}{\setlength{\unitlength}{0.7mm}
                     \begin{picture}(0,15)
                     \put(-20,0){\circle*{1.3}}
                     \put(-20,0){\vector(1,0){9}}
                     \put(-10,0){\circle*{1.3}}
                     \put(-10,0){\line(1,0){10}}
                     \put(0,0){\circle*{1.3}}
                     \put(0,0){\line(1,0){10}}
                     \put(10,0){\circle*{1.3}}
                     \put(0,0){\line(0,1){10}}
                     \put(0,10){\circle*{1.3}}
                     \put(-10,0){\vector(-1,1){9}}
                     \put(-20,10){\circle*{1.3}}
                     \put(-20,10){\vector(0,-1){9}}
                     \put(10,0){\line(1,0){10}}
                     \put(20,0){\circle*{1.3}}
                     \end{picture}    }  &\multicolumn{1}{|c|}{$\underline{\mod}(S_3)$}\\
\multicolumn{1}{|c|}{$2(x^7-x^5+x^4-x^3+x^2-1)$}  & \\
\multicolumn{1}{|c|}{   \setlength{\unitlength}{0.7mm}
                     \begin{picture}(0,25)
                     \put(-20,10){\circle*{1.3}}
                     \put(-20,10){\line(1,0){10}}
                     \put(-10,10){\circle*{1.3}}
                     \put(-10,10){\vector(1,-1){9}}
                     \put(0,0){\circle*{1.3}}
                     \put(0,0){\vector(0,1){20}}
                     \put(0,20){\circle*{1.3}}
                     \put(0,20){\vector(-1,-1){9}}
                     \put(0,20){\vector(1,-1){9}}
                     \put(10,10){\circle*{1.3}}
                     \put(10,10){\vector(-1,-1){9}}
                     \put(10,10){\line(1,0){10}}
                     \put(20,10){\circle*{1.3}}
                     \put(20,10){\line(1,0){10}}
                     \put(30,10){\circle*{1.3}}
                     \end{picture}   }  &\multicolumn{1}{|c|}{}\\
\multicolumn{1}{|c|}{$2(x^7-2x^5+4x^4-4x^3+2x^2-1)$}  &\multicolumn{1}{|c|}{}\\
\hline
\multicolumn{1}{|c|}{  \setlength{\unitlength}{0.7mm}
                     \begin{picture}(0,25)
                     \put(-20,10){\circle*{1.3}}
                     \put(-20,10){\vector(1,-1){9}}
                     \put(-10,0){\circle*{1.3}}
                     \put(-10,0){\vector(1,1){9}}
                     \put(0,10){\circle*{1.3}}
                     \put(0,10){\vector(-1,1){9}}
                     \put(-10,20){\circle*{1.3}}
                     \put(-10,20){\vector(-1,-1){9}}
                     \put(-10,20){\line(1,0){10}}
                     \put(0,20){\circle*{1.3}}
                     \put(0,10){\line(1,0){10}}
                     \put(10,10){\circle*{1.3}}
                     \put(10,10){\line(1,0){10}}
                     \put(20,10){\circle*{1.3}}
                     \end{picture}   }  &\multicolumn{1}{|c|}{$\underline{\mod}(S_4)$}\\
\multicolumn{1}{|c|}{$3(x^7-1)$}  &\multicolumn{1}{|c|}{}\\
\hline
\end{tabular}
\]

\[
\setlength{\tabcolsep}{2pt}
\renewcommand{\arraystretch}{1.2}
\begin{tabular}{|l|l|}
\hline

\multicolumn{1}{|c|}{ \setlength{\unitlength}{0.7mm}
                     \begin{picture}(0,25)
                     \put(-20,15){\circle*{1.3}}
                     \put(-20,15){\vector(0,-1){9}}
                     \put(-20,5){\circle*{1.3}}
                     \put(-20,5){\vector(2,-1){9}}
                     \put(-10,0){\circle*{1.3}}
                     \put(-10,0){\vector(1,1){9}}
                     \put(0,10){\circle*{1.3}}
                     \put(0,10){\vector(-1,1){9}}
                     \put(-10,20){\vector(-2,-1){9}}
                     \put(-10,20){\circle*{1.3}}
                     \put(0,10){\line(1,0){10}}
                     \put(10,10){\circle*{1.3}}
                     \put(10,10){\line(1,0){10}}
                     \put(20,10){\circle*{1.3}}
                     \end{picture} }  &\\
\multicolumn{1}{|c|}{$4(x^7+x^5-x^4+x^3-x^2-1)$}  &\multicolumn{1}{|c|}{$\underline{\mod}(S_5)$}\\

\multicolumn{1}{|c|}{\setlength{\unitlength}{0.7mm}
                     \begin{picture}(0,25)
                     \put(-20,10){\circle*{1.3}}
                     \put(-20,10){\vector(2,1){9}}
                     \put(-10,15){\circle*{1.3}}
                     \put(-10,15){\vector(0,-1){9}}
                     \put(-10,5){\circle*{1.3}}
                     \put(-10,5){\vector(-2,1){9}}
                     \put(-10,5){\vector(2,-1){9}}
                     \put(0,0){\circle*{1.3}}
                     \put(0,0){\vector(1,1){9}}
                     \put(10,10){\circle*{1.3}}
                     \put(10,10){\vector(-1,1){9}}
                     \put(0,20){\vector(-2,-1){9}}
                     \put(0,20){\circle*{1.3}}
                     \put(0,20){\line(1,0){10}}
                     \put(10,20){\circle*{1.3}}
                     \end{picture}  }  &\\
\multicolumn{1}{|c|}{$4(x^7+x^5-2x^4+2x^3-x^2-1)$}  &\multicolumn{1}{|c|}{}\\
\hline
\multicolumn{1}{|c|}{ \setlength{\unitlength}{0.7mm}
                     \begin{picture}(0,15)
                     \put(-20,0){\circle*{1.3}}
                     \put(-20,0){\vector(1,0){9}}
                     \put(-10,0){\circle*{1.3}}
                     \put(-10,0){\line(1,0){10}}
                     \put(10,0){\circle*{1.3}}
                     \put(10,0){\vector(0,1){9}}
                     \put(10,10){\circle*{1.3}}
                     \put(10,10){\vector(-1,-1){9}}
                     \put(-10,10){\circle*{1.3}}
                     \put(-10,0){\vector(-1,1){9}}
                     \put(-20,10){\circle*{1.3}}
                     \put(-20,10){\vector(0,-1){9}}
                      \put(-10,0){\line(0,1){10}}
                      \put(0,0){\circle*{1.3}}
                      \put(0,0){\vector(1,0){9}}
                     \end{picture}  }  &\multicolumn{1}{|c|}{}\\
\multicolumn{1}{|c|}{$4(x^7+x^6-x^5+x^4-x^3+x^2-x-1)$}  &\multicolumn{1}{|c|}{}\\
\multicolumn{1}{|c|}{ \setlength{\unitlength}{0.7mm}
                     \begin{picture}(0,25)
                     \put(-20,10){\circle*{1.3}}
                     \put(-20,10){\vector(1,1){9}}
                     \put(-20,10){\vector(1,-1){9}}
                     \put(-10,20){\circle*{1.3}}
                     \put(-10,0){\circle*{1.3}}
                     \put(-10,0){\vector(1,1){9}}
                     \put(0,10){\circle*{1.3}}
                     \put(-10,20){\vector(1,-1){9}}
                     \put(-10,20){\line(-1,0){10}}
                     \put(-20,20){\circle*{1.3}}
                     \put(0,10){\vector(-1,0){19}}
                     \put(0,10){\vector(1,-1){9}}
                     \put(10,0){\circle*{1.3}}
                     \put(10,0){\vector(0,1){19}}
                     \put(10,20){\circle*{1.3}}
                     \put(10,20){\vector(-1,-1){9}}
                     \end{picture}  }  &\multicolumn{1}{|c|}{}\\
\multicolumn{1}{|c|}{$4(x^7+x^6-x^5-x^4+x^3+x^2-x-1)$}  &\multicolumn{1}{|c|}{$\underline{\mod}(S_3) \coprod \underline{\mod}(S_3)$}\\
\multicolumn{1}{|c|}{ \setlength{\unitlength}{0.7mm}
                     \begin{picture}(0,25)
                     \put(-20,10){\circle*{1.3}}
                     \put(-20,10){\line(1,0){10}}
                     \put(-10,10){\circle*{1.3}}
                     \put(-10,10){\vector(1,-1){9}}
                     \put(0,0){\circle*{1.3}}
                     \put(0,0){\vector(0,1){20}}
                     \put(0,20){\circle*{1.3}}
                     \put(0,20){\vector(-1,-1){9}}
                     \put(0,20){\vector(1,-1){9}}
                     \put(10,10){\circle*{1.3}}
                     \put(10,10){\vector(-1,-1){9}}
                     \put(10,10){\vector(1,-1){9}}
                     \put(20,20){\circle*{1.3}}
                     \put(20,20){\vector(-1,-1){9}}
                     \put(20,0){\circle*{1.3}}
                     \put(20,0){\vector(0,1){19}}
                     \end{picture}  }  &\multicolumn{1}{|c|}{}\\
\multicolumn{1}{|c|}{$4(x^7+x^6-2x^5+2x^4-2x^3+2x^2-x-1)$}  &\multicolumn{1}{|c|}{}\\
\hline
\multicolumn{1}{|c|}{\setlength{\unitlength}{0.7mm}
                     \begin{picture}(0,25)
                     \put(-20,10){\circle*{1.3}}
                     \put(-20,10){\vector(1,-1){9}}
                     \put(-10,0){\circle*{1.3}}
                     \put(-10,0){\vector(1,0){9}}
                     \put(0,0){\circle*{1.3}}
                     \put(0,0){\vector(1,1){9}}
                     \put(10,10){\circle*{1.3}}
                     \put(10,10){\vector(-1,1){9}}
                     \put(0,20){\circle*{1.3}}
                     \put(0,20){\vector(-1,0){9}}
                     \put(-10,20){\circle*{1.3}}
                     \put(-10,20){\vector(-1,-1){9}}
                     \put(10,10){\line(1,0){10}}
                     \put(20,10){\circle*{1.3}}
                     \end{picture}  }  &\multicolumn{1}{|c|}{$\underline{\mod}(S_6)$}\\
\multicolumn{1}{|c|}{$5(x^7+x^5-x^4+x^3-x^2-1)$}  & \\
\hline
\multicolumn{1}{|c|}{  \setlength{\unitlength}{0.5mm}
                       \begin{picture}(20,50)(0,-20)
                       \put(10,-10){\circle*{2}}
                       \put(10,-10){\vector(0,1){19}}
                       \put(10,10){\circle*{2}}
                       \put(10,10){\vector(1,0){19}}
                       \put(30,10){\circle*{2}}
                       \put(30,10){\vector(0,-1){19}}
                       \put(30,-10){\circle*{2}}
                       \put(30,-10){\vector(-1,0){19}}

                      \put(20,20){\circle*{2}}
                      \put(20,20){\vector(-1,-1){9.7}}
                      \put(30,10){\vector(-1,1){9.7}}
                      \put(20,20){\vector(1,0){19}}
                      \put(40,20){\circle*{2}}
                      \put(40,20){\vector(-1,-1){9.7}}
                      \put(0,20){\circle*{2}}
                      \put(10,10){\vector(-1,1){9.7}}
                      \put(0,20){\vector(1,0){19}}

                      \put(7,-15){1}
                      \put(31,-15){4}
                      \put(7,6){2}
                      \put(31,6){5}
                      \put(-3,21){3}
                      \put(18,21){6}
                      \put(40,21){7}
                      \put(15,-18){$A_{104}$}
                       \end{picture}}  &\multicolumn{1}{|c|}{$\underline{\mod}(S_7)$ }\\
\multicolumn{1}{|c|}{$6(x^7+x^5-x^2-1)$}  &\\
\hline
\multicolumn{1}{|c|}{    \setlength{\unitlength}{0.7mm}
                     \begin{picture}(0,25)
                     \put(-20,10){\circle*{1.3}}
                     \put(-20,10){\vector(1,-1){9}}
                     \put(-10,0){\circle*{1.3}}
                     \put(-10,0){\vector(1,1){9}}
                     \put(0,10){\circle*{1.3}}
                     \put(0,10){\vector(-1,1){9}}
                     \put(-10,20){\circle*{1.3}}
                     \put(-10,20){\vector(-1,-1){9}}
                     \put(-10,20){\line(-1,0){10}}
                     \put(-20,20){\circle*{1.3}}
                     \put(0,10){\vector(1,-1){9}}
                     \put(10,0){\circle*{1.3}}
                     \put(10,0){\vector(0,1){19}}
                     \put(10,20){\circle*{1.3}}
                     \put(10,20){\vector(-1,-1){9}}
                     \end{picture}   }  &\multicolumn{1}{|c|}{$\underline{\mod}(S_3)\coprod \underline{\mod}(S_4)$ }\\
\multicolumn{1}{|c|}{$6(x^7+x^6-x^4+x^3-x-1)$}  &\\
\hline
\multicolumn{1}{|c|}{  \setlength{\unitlength}{0.7mm}
                     \begin{picture}(0,25)
                     \put(-20,15){\circle*{1.3}}
                     \put(-20,15){\vector(0,-1){9}}
                     \put(-20,5){\circle*{1.3}}
                     \put(-20,5){\vector(2,-1){9}}
                     \put(-10,0){\circle*{1.3}}
                     \put(-10,0){\vector(1,1){9}}
                     \put(0,10){\circle*{1.3}}
                     \put(0,10){\vector(-1,1){9}}
                     \put(-10,20){\vector(-2,-1){9}}
                     \put(-10,20){\circle*{1.3}}
                     \put(0,10){\vector(1,-1){9}}
                     \put(10,0){\circle*{1.3}}
                     \put(10,0){\vector(0,1){19}}
                     \put(10,20){\circle*{1.3}}
                     \put(10,20){\vector(-1,-1){9}}
                     \end{picture} }  &\multicolumn{1}{|c|}{$\underline{\mod}(S_3) \coprod \underline{\mod}(S_5)$ }\\
\multicolumn{1}{|c|}{$8(x^7+x^6+x^5-x^4+x^3-x^2-x-1)$}  &\\
\hline
\end{tabular}
\]

\[\setlength{\tabcolsep}{2pt}
\renewcommand{\arraystretch}{1.2}
\begin{tabular}{|l|l|}
\hline
\multicolumn{2}{|c|}{Singularity equivalence classes of Dynkin type $E_8$}\\
\hline
\multicolumn{1}{|c|}{The quiver and associated polynomial of $A$}  &\multicolumn{1}{|c|}{$\underline{\Gproj}(A)$}\\
\hline
\multicolumn{1}{|c|}{ \setlength{\unitlength}{0.7mm}
                     \begin{picture}(10,15)
                     \put(-20,0){\circle*{1.3}}
                     \put(-20,0){\line(1,0){10}}
                     \put(-10,0){\circle*{1.3}}
                     \put(-10,0){\line(1,0){10}}
                     \put(0,0){\circle*{1.3}}
                     \put(0,0){\line(1,0){10}}
                     \put(10,0){\circle*{1.3}}
                     \put(10,0){\line(1,0){10}}
                     \put(20,0){\circle*{1.3}}
                     \put(0,0){\line(0,1){10}}
                     \put(0,10){\circle*{1.3}}
                     \put(20,0){\line(1,0){10}}
                     \put(30,0){\circle*{1.3}}
                     \put(30,0){\line(1,0){10}}
                     \put(40,0){\circle*{1.3}}
                     \end{picture}}  &\multicolumn{1}{|c|}{$0$}\\
\multicolumn{1}{|c|}{$x^8-x^7+x^5-x^4+x^3-x+1$}  &  \\
\hline
\multicolumn{1}{|c|}{    \setlength{\unitlength}{0.7mm}
                     \begin{picture}(0,15)
                     \put(-20,0){\circle*{1.3}}
                     \put(-20,0){\vector(1,0){9}}
                     \put(-10,0){\circle*{1.3}}
                     \put(-10,0){\line(1,0){10}}
                     \put(0,0){\circle*{1.3}}
                     \put(0,0){\line(1,0){10}}
                     \put(10,0){\circle*{1.3}}
                     \put(-10,0){\line(0,1){10}}
                     \put(-10,10){\circle*{1.3}}
                     \put(-10,0){\vector(-1,1){9}}
                     \put(-20,10){\circle*{1.3}}
                     \put(-20,10){\vector(0,-1){9}}
                     \put(10,0){\line(1,0){10}}
                     \put(20,0){\circle*{1.3}}
                     \put(20,0){\line(1,0){10}}
                     \put(30,0){\circle*{1.3}}
                     \end{picture} }  &\multicolumn{1}{|c|}{}\\
\multicolumn{1}{|c|}{$2(x^8-x^6+2x^5-2x^4+2x^3-x^2+1)$}  &\multicolumn{1}{|c|}{}\\
\multicolumn{1}{|c|}{\setlength{\unitlength}{0.7mm}
                     \begin{picture}(0,15)
                     \put(-20,0){\circle*{1.3}}
                     \put(-20,0){\vector(1,0){9}}
                     \put(-10,0){\circle*{1.3}}
                     \put(-10,0){\line(1,0){10}}
                     \put(0,0){\circle*{1.3}}
                     \put(0,0){\line(1,0){10}}
                     \put(10,0){\circle*{1.3}}
                     \put(10,0){\line(0,1){10}}
                     \put(10,10){\circle*{1.3}}
                     \put(-10,0){\vector(-1,1){9}}
                     \put(-20,10){\circle*{1.3}}
                     \put(-20,10){\vector(0,-1){9}}
                     \put(10,0){\line(1,0){10}}
                     \put(20,0){\circle*{1.3}}
                     \put(20,0){\line(1,0){10}}
                     \put(30,0){\circle*{1.3}}
                     \end{picture}    }  &\multicolumn{1}{|c|}{$\underline{\mod}(S_3)$}\\
\multicolumn{1}{|c|}{$2(x^8-x^6+x^5+x^3-x^2+1)$}  &\\
\multicolumn{1}{|c|}{  \setlength{\unitlength}{0.7mm}
                     \begin{picture}(0,25)
                     \put(-20,10){\circle*{1.3}}
                     \put(-20,10){\line(1,0){10}}
                     \put(-10,10){\circle*{1.3}}
                     \put(-10,10){\vector(1,-1){9}}
                     \put(0,0){\circle*{1.3}}
                     \put(0,0){\vector(0,1){20}}
                     \put(0,20){\circle*{1.3}}
                     \put(0,20){\vector(-1,-1){9}}
                     \put(0,20){\vector(1,-1){9}}
                     \put(10,10){\circle*{1.3}}
                     \put(10,10){\vector(-1,-1){9}}
                     \put(10,10){\line(1,0){10}}
                     \put(20,10){\circle*{1.3}}
                     \put(20,10){\line(1,0){10}}
                     \put(30,10){\circle*{1.3}}
                     \put(30,10){\line(1,0){10}}
                     \put(40,10){\circle*{1.3}}
                     \end{picture}  }  &\multicolumn{1}{|c|}{}\\
\multicolumn{1}{|c|}{$2(x^8-2x^6+4x^5-4x^4+4x^3-2x^2+1)$}  &\multicolumn{1}{|c|}{}\\
\hline
\multicolumn{1}{|c|}{ \setlength{\unitlength}{0.7mm}
                     \begin{picture}(0,25)
                     \put(-20,10){\circle*{1.3}}
                     \put(-20,10){\vector(1,-1){9}}
                     \put(-10,0){\circle*{1.3}}
                     \put(-10,0){\vector(1,1){9}}
                     \put(0,10){\circle*{1.3}}
                     \put(0,10){\vector(-1,1){9}}
                     \put(-10,20){\circle*{1.3}}
                     \put(-10,20){\vector(-1,-1){9}}
                     \put(-10,20){\line(-1,0){10}}
                     \put(-20,20){\circle*{1.3}}
                     \put(0,10){\line(1,0){10}}
                     \put(10,10){\circle*{1.3}}
                     \put(10,10){\line(1,0){10}}
                     \put(20,10){\circle*{1.3}}
                     \put(20,10){\line(1,0){10}}
                     \put(30,10){\circle*{1.3}}
                     \end{picture}   }  &\multicolumn{1}{|c|}{$\underline{\mod}(S_4)$}\\
\multicolumn{1}{|c|}{$3(x^8+x^4+1)$}  &  \\
\hline
\multicolumn{1}{|c|}{\setlength{\unitlength}{0.7mm}
                     \begin{picture}(0,25)
                     \put(-20,15){\circle*{1.3}}
                     \put(-20,15){\vector(0,-1){9}}
                     \put(-20,5){\circle*{1.3}}
                     \put(-20,5){\vector(2,-1){9}}
                     \put(-10,0){\circle*{1.3}}
                     \put(-10,0){\vector(1,1){9}}
                     \put(0,10){\circle*{1.3}}
                     \put(0,10){\vector(-1,1){9}}
                     \put(-10,20){\vector(-2,-1){9}}
                     \put(-10,20){\circle*{1.3}}
                     \put(0,10){\line(1,0){10}}
                     \put(10,10){\circle*{1.3}}
                     \put(10,10){\line(1,0){10}}
                     \put(20,10){\circle*{1.3}}
                     \put(20,10){\line(1,0){10}}
                     \put(30,10){\circle*{1.3}}
                     \end{picture}}  &\multicolumn{1}{|c|}{}\\
\multicolumn{1}{|c|}{$4(x^8+x^6-x^5+2x^4-x^3+x^2+1)$}  &\multicolumn{1}{|c|}{$\underline{\mod}(S_5)$}\\
\multicolumn{1}{|c|}{ \setlength{\unitlength}{0.7mm}
                     \begin{picture}(0,25)
                     \put(-20,10){\circle*{1.3}}
                     \put(-20,10){\vector(2,1){9}}
                     \put(-10,15){\circle*{1.3}}
                     \put(-10,15){\vector(0,-1){9}}
                     \put(-10,5){\circle*{1.3}}
                     \put(-10,5){\vector(-2,1){9}}
                     \put(-10,5){\vector(2,-1){9}}
                     \put(0,0){\circle*{1.3}}
                     \put(0,0){\vector(1,1){9}}
                     \put(10,10){\circle*{1.3}}
                     \put(10,10){\vector(-1,1){9}}
                     \put(0,20){\vector(-2,-1){9}}
                     \put(0,20){\circle*{1.3}}
                     \put(0,20){\line(1,0){10}}
                     \put(10,20){\circle*{1.3}}
                     \put(-30,10){\circle*{1.3}}
                     \put(-30,10){\line(1,0){10}}
                     \end{picture}    }  &\multicolumn{1}{|c|}{}\\
\multicolumn{1}{|c|}{$4(x^8+x^6-2x^5+4x^4-2x^3+x^2+1)$}  &\multicolumn{1}{|c|}{}\\
\hline
\multicolumn{1}{|c|}{ \setlength{\unitlength}{0.7mm}
                     \begin{picture}(0,15)
                     \put(-20,0){\circle*{1.3}}
                     \put(-20,0){\vector(1,0){9}}
                     \put(-10,0){\circle*{1.3}}
                     \put(-10,0){\line(1,0){10}}
                     \put(20,0){\circle*{1.3}}
                     \put(20,0){\vector(0,1){9}}
                     \put(20,10){\circle*{1.3}}
                     \put(20,10){\vector(-1,-1){9}}
                     \put(-10,10){\circle*{1.3}}
                     \put(-10,0){\vector(-1,1){9}}
                     \put(-20,10){\circle*{1.3}}
                     \put(-20,10){\vector(0,-1){9}}
                      \put(-10,0){\line(0,1){10}}
                      \put(0,0){\circle*{1.3}}
                      \put(0,0){\line(1,0){10}}
                      \put(10,0){\circle*{1.3}}
                      \put(10,0){\vector(1,0){9}}
                     \end{picture}}  &\multicolumn{1}{|c|}{}\\
\multicolumn{1}{|c|}{$4(x^8+x^7-x^6+x^5+x^3-x^2+x+1)$}  & \\
\multicolumn{1}{|c|}{\setlength{\unitlength}{0.7mm}
                     \begin{picture}(0,15)
                     \put(-20,0){\circle*{1.3}}
                     \put(-20,0){\vector(1,0){9}}
                     \put(-10,0){\circle*{1.3}}
                     \put(-10,0){\vector(1,0){9}}
                     \put(0,0){\circle*{1.3}}
                     \put(0,0){\vector(0,1){9}}
                     \put(0,10){\circle*{1.3}}
                     \put(0,10){\vector(-1,-1){9}}

                     \put(-10,0){\vector(-1,1){9}}
                     \put(-20,10){\circle*{1.3}}
                     \put(-20,10){\vector(0,-1){9}}
                      \put(0,0){\line(1,0){10}}
                      \put(10,0){\circle*{1.3}}
                       \put(10,0){\line(1,0){10}}
                      \put(20,0){\circle*{1.3}}
                      \put(0,0){\line(1,1){10}}
                      \put(10,10){\circle*{1.3}}
                     \end{picture} }  &\multicolumn{1}{|c|}{}\\
\multicolumn{1}{|c|}{$4(x^8+x^7-x^6+2x^4-x^2+x+1)$}  &\multicolumn{1}{|c|}{$\underline{\mod}(S_3) \coprod \underline{\mod}(S_3)$}\\
\multicolumn{1}{|c|}{  \setlength{\unitlength}{0.7mm}
                     \begin{picture}(0,25)
                     \put(-20,10){\circle*{1.3}}
                     \put(-20,10){\line(1,0){10}}
                     \put(-10,10){\circle*{1.3}}
                     \put(-10,10){\vector(1,-1){9}}
                     \put(0,0){\circle*{1.3}}
                     \put(0,0){\vector(0,1){20}}
                     \put(0,20){\circle*{1.3}}
                     \put(0,20){\vector(-1,-1){9}}
                     \put(0,20){\vector(1,-1){9}}
                     \put(10,10){\circle*{1.3}}
                     \put(10,10){\vector(-1,-1){9}}
                     \put(10,10){\line(1,0){10}}
                     \put(20,10){\circle*{1.3}}
                     \put(20,10){\vector(1,-1){9}}
                     \put(30,0){\circle*{1.3}}
                     \put(30,0){\vector(0,1){19}}
                     \put(30,20){\circle*{1.3}}
                     \put(30,20){\vector(-1,-1){9}}
                     \end{picture} }  &\multicolumn{1}{|c|}{}\\
\multicolumn{1}{|c|}{$4(x^8+x^7-2x^6+2x^5+2x^3-2x^2+x+1)$}  &\multicolumn{1}{|c|}{}\\
\hline
\end{tabular}
\]
\[\setlength{\tabcolsep}{2pt}
\renewcommand{\arraystretch}{1.2}
\begin{tabular}{|l|l|}
\hline
\multicolumn{1}{|c|}{ \setlength{\unitlength}{0.7mm}
                     \begin{picture}(0,25)
                     \put(-20,10){\circle*{1.3}}
                     \put(-20,10){\vector(1,-1){9}}
                     \put(-10,0){\circle*{1.3}}
                     \put(-10,0){\vector(1,0){9}}
                     \put(0,0){\circle*{1.3}}
                     \put(0,0){\vector(1,1){9}}
                     \put(10,10){\circle*{1.3}}
                     \put(10,10){\vector(-1,1){9}}
                     \put(0,20){\circle*{1.3}}
                     \put(0,20){\vector(-1,0){9}}
                     \put(-10,20){\circle*{1.3}}
                     \put(-10,20){\vector(-1,-1){9}}
                     \put(10,10){\line(1,0){10}}
                     \put(20,10){\circle*{1.3}}
                     \put(-30,10){\circle*{1.3}}
                     \put(-30,10){\line(1,0){10}}
                     \end{picture} }  &\multicolumn{1}{|c|}{$\underline{\mod}(S_6)$}\\
\multicolumn{1}{|c|}{$5(x^8+x^6+x^4+x^2+1)$}  &\\
\hline
\multicolumn{1}{|c|}{ \setlength{\unitlength}{0.7mm}
                     \begin{picture}(0,25)
                     \put(-20,15){\circle*{1.3}}
                     \put(-20,15){\vector(0,-1){9}}
                     \put(-20,5){\circle*{1.3}}
                     \put(-20,5){\vector(2,-1){9}}
                     \put(-10,0){\circle*{1.3}}
                     \put(-10,0){\vector(1,0){9}}
                     \put(0,20){\circle*{1.3}}
                     \put(0,20){\vector(-1,0){9}}
                     \put(-10,20){\vector(-2,-1){9}}
                     \put(-10,20){\circle*{1.3}}
                     \put(10,10){\line(1,0){10}}
                     \put(20,10){\circle*{1.3}}
                     \put(0,0){\circle*{1.3}}
                     \put(0,0){\vector(1,1){9}}
                     \put(10,10){\circle*{1.3}}
                     \put(10,10){\vector(-1,1){9}}
                     \put(-10,10){$A_{15}$}
                     \end{picture} }  &\multicolumn{1}{|c|}{$\underline{\mod}(S_7)$ }\\
\multicolumn{1}{|c|}{$6(x^8+x^6+x^5+x^3+x^2+1)$}  & \\
\hline
\multicolumn{1}{|c|}{ \setlength{\unitlength}{0.7mm}
                     \begin{picture}(0,25)
                     \put(-20,10){\circle*{1.3}}
                     \put(-20,10){\vector(1,-1){9}}
                     \put(-10,0){\circle*{1.3}}
                     \put(-10,0){\vector(1,1){9}}
                     \put(0,10){\circle*{1.3}}
                     \put(0,10){\vector(-1,1){9}}
                     \put(-10,20){\circle*{1.3}}
                     \put(-10,20){\vector(-1,-1){9}}
                     \put(-10,20){\line(-1,0){10}}
                     \put(-20,20){\circle*{1.3}}
                     \put(10,10){\vector(1,-1){9}}
                     \put(20,0){\circle*{1.3}}
                     \put(20,0){\vector(0,1){19}}
                     \put(20,20){\circle*{1.3}}
                     \put(20,20){\vector(-1,-1){9}}
                     \put(0,10){\line(1,0){10}}
                     \put(10,10){\circle*{1.3}}
                     \end{picture}   }  &\multicolumn{1}{|c|}{$\underline{\mod}(S_3) \coprod \underline{\mod}(S_4)$ }\\
\multicolumn{1}{|c|}{$6(x^8+x^7+2x^4+x+1)$}  & \\
\hline
\multicolumn{1}{|c|}{\setlength{\unitlength}{0.7mm}
                     \begin{picture}(0,25)
                     \put(-20,15){\circle*{1.3}}
                     \put(-20,15){\vector(0,-1){9}}
                     \put(-20,5){\circle*{1.3}}
                     \put(-20,5){\vector(2,-1){9}}
                     \put(-10,0){\circle*{1.3}}
                     \put(-10,0){\vector(1,1){9}}
                     \put(0,10){\circle*{1.3}}
                     \put(0,10){\vector(-1,1){9}}
                     \put(-10,20){\vector(-2,-1){9}}
                     \put(-10,20){\circle*{1.3}}
                     \put(10,10){\vector(1,-1){9}}
                     \put(20,0){\circle*{1.3}}
                     \put(20,0){\vector(0,1){19}}
                     \put(20,20){\circle*{1.3}}
                     \put(20,20){\vector(-1,-1){9}}
                     \put(0,10){\line(1,0){10}}
                     \put(10,10){\circle*{1.3}}
                     \end{picture}  }  &\multicolumn{1}{|c|}{$\underline{\mod}(S_3) \coprod \underline{\mod}(S_5)$ }\\
\multicolumn{1}{|c|}{$8(x^8+x^7+x^6+2x^4+x^2+x+1)$}  & \\
\hline
\multicolumn{1}{|c|}{   \setlength{\unitlength}{0.7mm}
                     \begin{picture}(0,15)
                     \put(-20,0){\circle*{1.3}}
                     \put(-20,0){\vector(1,0){9}}
                     \put(-10,0){\circle*{1.3}}
                     \put(-10,0){\line(1,0){10}}
                     \put(10,0){\circle*{1.3}}
                     \put(10,0){\vector(0,1){9}}
                     \put(10,10){\circle*{1.3}}
                     \put(10,10){\vector(-1,-1){9}}
                     \put(-10,10){\circle*{1.3}}
                     \put(-10,0){\vector(-1,1){9}}
                     \put(-20,10){\circle*{1.3}}
                     \put(-20,10){\vector(0,-1){9}}
                      \put(-10,0){\line(0,1){10}}
                      \put(0,0){\circle*{1.3}}
                      \put(0,0){\vector(0,1){9}}
                      \put(0,10){\circle*{1.3}}
                      \put(0,10){\vector(-1,-1){9}}
                      \put(0,0){\vector(1,0){9}}
                     \end{picture}  }  &\multicolumn{1}{|c|}{$\coprod_3 \underline{\mod}(S_3)$  }\\
\multicolumn{1}{|c|}{$8(x^8+2x^7+2x^4+2x+1)$}  & \\
\hline

\end{tabular}
\]

In fact, if $A$ is the representative of type $E_7$ with the associated polynomial
$6(x^7+x^5-x^2-1)$, which is denoted by $A_{104}$ in the second table above, then
\cite[Theorem 4.1]{Chen1} implies that $A_{104}$ is singularity equivalent to $A_{319}$ which is a cluster-tilted algebra of type $E_8$ \cite{GP}, whose quiver is as the following diagram shows. Since the associated polynomial of $A_{319}$ is $6(x^8+x^6+x^5+x^3+x^2+1)$, or seeing the supplementary material \cite{La2}, we get that $A_{319}$ is derived equivalent to $A_{15}$ in the third table above.
So $\underline{\Gproj}A_{104}\simeq \underline{\Gproj}A_{319}\simeq\underline{\Gproj}A_{15}\simeq \underline{\mod}S_7$.
Furthermore, the representatives of the other derived equivalence classes can be viewed as gluing of rooted quivers of type $A$
to certain skeleta which is a cluster-tilted algebra of type $A$ or $D$, their singularity categories can be obtained by Theorem \ref{theorem rooted quiver to type A} and Theorem \ref{theorem singularity category of D} directly.

\setlength{\unitlength}{0.7mm}
\begin{center}
\begin{picture}(100,60)(0,-20)
\put(30,0){\circle*{2}}
\put(30,0){\vector(0,1){19}}
\put(30,20){\circle*{2}}
\put(30,20){\vector(1,0){19}}
\put(50,20){\circle*{2}}
\put(50,20){\vector(0,-1){19}}
\put(50,0){\circle*{2}}
\put(50,0){\vector(-1,0){19}}

\put(40,30){\circle*{2}}
\put(40,30){\vector(-1,-1){9.7}}
\put(50,20){\vector(-1,1){9.7}}
\put(40,30){\vector(1,0){19}}
\put(60,30){\circle*{2}}
\put(60,30){\vector(-1,-1){9.7}}
\put(20,30){\circle*{2}}
\put(30,20){\vector(-1,1){9.7}}
\put(20,30){\vector(1,0){19}}

\put(27,-5){1}
\put(51,-5){4}
\put(27,16){2}
\put(51,16){5}
\put(17,31){3}
\put(38,31){6}
\put(60,31){7}

\put(60,10){\circle*{2}}
\put(60,10){\vector(0,1){19}}
\put(61,7){8}

\put(35,-15){$A_{319}$}

\end{picture}
\end{center}

From the above tables, we get the following result.

\begin{theorem}\label{theorem singularity category of E}
Two cluster-tilted algebras of Dynkin type $E$ are singularity equivalent if and only if the Cartan matrices of the algebras have the same determinant and the
same trace of their asymmetry matrices.
\end{theorem}
\begin{proof}
We only need prove that for any cluster-tilted algebra $A$ of Dynkin type $E$, its singularity category is uniquely determined by
$\det(C_A)$ and the traces $\chi_A$ of its asymmetry matrix.
By a simple observation of the three tables above, we can get the following table, and from it, our desired statement follows immediately.

\[
\setlength{\tabcolsep}{2pt}
\renewcommand{\arraystretch}{1.2}
\begin{tabular}{|l|l|l|}
\hline
\multicolumn{1}{|c|}{\quad $\det(C_A)$ \quad } &\multicolumn{1}{|c|}{\quad $\chi_A$ \quad} &\multicolumn{1}{|c|}{$\underline{\Gproj}(A)$}\\
\hline
\multicolumn{1}{|c|}{1} &\multicolumn{1}{|c|}{0} &\multicolumn{1}{|c|}{$0$}\\
\hline
\multicolumn{1}{|c|}{2} &\multicolumn{1}{|c|}{0} &\multicolumn{1}{|c|}{$\underline{\mod}(S_3)$}\\
\hline
\multicolumn{1}{|c|}{3} &\multicolumn{1}{|c|}{0} &\multicolumn{1}{|c|}{$\underline{\mod}(S_4)$}\\
\hline
\multicolumn{1}{|c|}{4} &\multicolumn{1}{|c|}{0} &\multicolumn{1}{|c|}{$\underline{\mod}(S_5)$}\\
\hline
\multicolumn{1}{|c|}{4} &\multicolumn{1}{|c|}{1} &\multicolumn{1}{|c|}{$\underline{\mod} S_{3}\coprod \underline{\mod} S_{3}$}\\
\hline
\multicolumn{1}{|c|}{5} &\multicolumn{1}{|c|}{0} &\multicolumn{1}{|c|}{$\underline{\mod}(S_6)$}\\
\hline
\multicolumn{1}{|c|}{6} &\multicolumn{1}{|c|}{0} &\multicolumn{1}{|c|}{$\underline{\mod}(S_7)$}\\
\hline
\multicolumn{1}{|c|}{6} &\multicolumn{1}{|c|}{1} &\multicolumn{1}{|c|}{$\underline{\mod} S_{3}\coprod \underline{\mod} S_{4}$}\\
\hline
\multicolumn{1}{|c|}{8} &\multicolumn{1}{|c|}{1} &\multicolumn{1}{|c|}{$\underline{\mod} S_{3}\coprod\underline{\mod}(S_5)$}\\
\hline
\multicolumn{1}{|c|}{8} &\multicolumn{1}{|c|}{2} &\multicolumn{1}{|c|}{$\underline{\mod} S_{3}\coprod \underline{\mod} S_{3}\coprod \underline{\mod} S_{3}$}\\
\hline
\end{tabular}
\]
\end{proof}

\begin{corollary}
For any cluster-tilted algebra $A$ of Dynkin type $E$, there exists some selfinjective Nakayama
algebra $\Lambda$ such that $\underline{\Gproj}(A)\simeq \underline{\mod}(\Lambda)$. In particular, the number of the nontrivial connected components of $\Lambda$ is
$\chi_A+1$.
\end{corollary}
\begin{proof}
Note that the Auslander Reiten quiver of $\underline{\mod}S_i$ is connected for any $i\geq 3$. We get that the connected components of $\Lambda$ correspond to the connected components of $\underline{\mod}(\Lambda)$ and also $\underline{\Gproj}(A)$
since they are invariant under the triangulated equivalences. So the result follows from the table above in the proof of Theorem \ref{theorem singularity category of E} easily.
\end{proof}

\end{document}